\documentclass{amsart}

\usepackage{graphicx}
\usepackage[margin=1in]{geometry}
\usepackage{amssymb}
\usepackage{amsfonts}
\usepackage{amsthm}
\usepackage{amsmath}
\usepackage{hyperref}
\usepackage{dsfont}
\usepackage{mathtools}

\newcommand{\FF}{\mathbb{F}}
\newcommand{\GAP}{\mathrm{GAP}}
\newcommand{\EE}{\mathrm{E}}
\newcommand{\PP}{\mathrm{P}}
\newcommand{\Anq}{\mathcal{A}_{n,q}}
\newcommand{\Mnq}{\mathcal{M}_{n,q}}
\newcommand{\indic}{\mathds{1}}
\newcommand{\Tr}{\operatorname{tr}}

\mathtoolsset{showonlyrefs}

\author{Ofir Gorodetsky}
\address{Raymond and Beverly Sackler School of Mathematical Sciences, Tel Aviv University, Tel Aviv 69978, Israel}
\email{ofir.goro@gmail.com}

\author{Will Sawin}
\address{Department of Mathematics, Columbia University, New York, NY 10027, USA}
\email{sawin@math.columbia.edu}

\newtheorem{theorem}{Theorem}
\newtheorem{lemma}{Lemma}  
\newtheorem{proposition}{Proposition}
\newtheorem{corollary}{Corollary}
\newtheorem{conjecture}{Conjecture} 

\theoremstyle{remark}
\newtheorem{remark}{Remark}

\numberwithin{equation}{section}

\begin{document}
	 \title{Correlation of arithmetic functions over $\FF_q[T]$}

\begin{abstract}
For a fixed polynomial $\Delta$, we study the number of polynomials $f$ of degree $n$ over $\mathbb F_q$ such that $f$ and $f+\Delta$ are both irreducible, an $\mathbb F_q[T]$-analogue of the twin primes problem. In the large-$q$ limit, we obtain a lower-order term for this count if we consider non-monic polynomials, which depends on $\Delta$ in a manner which is consistent with the Hardy-Littlewood Conjecture. We obtain a saving of $q$ if we consider monic polynomials only and $\Delta$ is a scalar. To do this, we use symmetries of the problem to get for free a small amount of averaging in $\Delta$. This allows us to obtain additional saving from equidistribution results for $L$-functions. We do all this in a combinatorial framework that applies to more general arithmetic functions than the indicator function of irreducibles, including the M\"{o}bius function and divisor functions.

\end{abstract}
	\maketitle
\section{Introduction}
Let $q$ be a prime power, $\FF_q$ be the finite field with $q$ elements and $\FF_q[T]$ be the polynomial ring over $\FF_q$ in indeterminate $T$.
\begin{itemize}
\item We denote by $\Anq$ the set of polynomials of degree $n$ in $\FF_q[T]$ and by $\mathcal{A}_q = \cup_{n \ge 0} \Anq$ the set of non-zero polynomials in $\FF_q[T]$.
\item Similarly, we denote by $\Mnq\subseteq \Anq$ the set of monic polynomials of degree $n$ in $\FF_q[T]$ and by $\mathcal{M}_q = \cup_{n \ge 0} \Mnq$ the set of all monic polynomials in $\FF_q[T]$. 
\item We denote by $\mathcal{P}_{n,q}\subseteq \Mnq$ the set of monic irreducible polynomials of degree $n$ in $\FF_q[T]$ and by $\mathcal{P}_q = \cup_{n \ge 1} \mathcal{P}_{n,q}$ the set of all monic irreducible polynomials in $\FF_q[T]$.
\item Given $\Delta \in \mathcal{A}_q$, we denote by $a_{\Delta,q}$ the number of distinct roots of $\Delta$ in $\FF_q$:
\begin{equation*}
a_{\Delta,q} = \# \{ a \in \FF_q : \Delta(a)=0\}.
\end{equation*}
\end{itemize}

There are several important arithmetic functions $\FF_q[T]\setminus \{0\} \to \mathbb{C}$ which are used to study polynomials. Two well-known examples include the von Mangoldt function $\Lambda_q$, defined on monics by
\begin{equation*}
\Lambda_q(f) =\begin{cases} \deg P & \text{if }f = P^k, \text{ where }P \in \mathcal{P}_q,\, k >0, \\ 0 & \text{otherwise,} \end{cases}
\end{equation*}
and the M\"obius function $\mu_q$, defined on monics by
\begin{equation*}
\mu_q(f) = \begin{cases} (-1)^k & \text{if }f= \prod_{i=1}^{k} P_i, \text{ where } P_i\in \mathcal{P}_q \text{ are distinct}, \\ 0 & \text{otherwise.} \end{cases}
\end{equation*}
We extend these functions to non-monics by $\alpha(c\cdot f)=\alpha(f)$ for all $c \in \FF_q^{\times}, f \in \mathcal{M}_q$. The mean value of $\Lambda_q$ has a well-known closed form \cite[Prop.~2.1]{rosen2002}:
\begin{equation*}
\frac{\sum_{f \in \Anq} \Lambda_q(f)}{q^n(q-1)} = \frac{\sum_{f \in \Mnq} \Lambda_q(f)}{q^n} = 1,
\end{equation*}
and this is an analogue of the Prime Number Theorem. For $\Delta \in \FF_q[T] \setminus \{ 0\}$, the asymptotics of the mean values of $\Lambda_q(f) \Lambda_q(f+\Delta)$ over either $f\in \Anq$ or $f\in \Mnq$, that is,
\begin{equation*}
\frac{\sum_{f \in \Anq} \Lambda_q(f) \Lambda_q(f+\Delta)}{q^n(q-1)} \text{ or } \frac{\sum_{f \in \Mnq} \Lambda_q(f) \Lambda_q(f+\Delta)}{q^n},
\end{equation*}
are much less understood. Such asymptotic gives us quantitative information about pairs of primes in $\FF_q[T]$ whose difference is $\Delta$. The Hardy-Littlewood Conjecture predicts precise asymptotics for these mean values when $q^n \to \infty$, see \S\ref{twprdisc}. In this paper, we are interested in improving asymptotic results for the mean value of $\Lambda_q(f) \Lambda_q(f+\Delta)$, and other `shifted correlation' sums, in the limit $q \to \infty$. This limit is also known in the literature as the `large-$q$ limit' or `large finite field limit'. In particular, we think of $n$ as fixed, and $\Delta \in \FF_q[T] \setminus \{0\}$ is always of degree less than $n$. For results on twin primes in $\FF_q[T]$ where $q$ is fixed and $n \to \infty$, see the work of Castillo, Hall, Lemke Oliver, Pollack and Thompson \cite{castillo2015} and the results of the second author and Shusterman \cite{sawin2019} obtained while this article was in press.

Pollack \cite[Thm.~2]{pollack2008}, Bender and Pollack \cite[Thm.~1.3]{bender2009}, Bary-Soroker \cite[Thm.~1.1]{bary2014} and Carmon \cite[\S6]{carmon2015} have shown that 
\begin{equation}\label{twincornonm}
\frac{\sum_{f \in \Mnq} \Lambda_q(f) \Lambda_q(f+\Delta)}{q^n} - 1= O_{n}\Big(\frac{1}{\sqrt{q}}\Big),
\end{equation}
from which it follows by the first part of Lemma~\ref{monicnonmoniclem} below that
\begin{equation}\label{twincorm}
\frac{\sum_{f \in \Anq} \Lambda_q(f) \Lambda_q(f+\Delta)}{q^n(q-1)} -1= O_{n}\Big(\frac{1}{\sqrt{q}}\Big)
\end{equation}
as well. The proofs of \eqref{twincornonm} involve the calculation of the Galois groups of certain generic polynomials, and an application of a function field analogue of Chebotarev Density Theorem. These methods give an implied constant of order $n!^2$ (see the statement of \cite[Thm.~1.3]{bender2009}).

Using methods from $L$-functions, Pollack has shown that \cite[Thm.~1]{pollack2008polynomial}
\begin{equation}\label{pollackres}
\frac{\sum_{f \in \Anq} \Lambda_q(f) \Lambda_q(f+\Delta)}{q^n(q-1)} -1= O\Big(\frac{|\Delta|}{\phi(\Delta)}\frac{n^2}{q} \Big),
\end{equation}
where $\phi$ is Euler's totient function. As $\frac{|\Delta|}{\phi(\Delta)} \le 2^{\deg \Delta} \le 2^{n-1}$, the dependence of the error term in \eqref{pollackres} on $n$ is much better than in \eqref{twincorm} and \eqref{twincornonm}. More importantly for us, the dependence on $q$ is better in \eqref{pollackres} compared to \eqref{twincorm} and \eqref{twincornonm}, and in fact the Hardy-Littlewood Conjecture predicts that the left hand side of \eqref{pollackres} is, in general, $\Theta_n(\frac{1}{q})$, see \S\ref{twprdisc}. Thus, the power of $q$ appearing in Pollack's result is the best possible.

When $\Delta=1$, Keating and Roditty-Gershon \cite{keating2016,keating2018} have improved \eqref{pollackres} in the large-$q$ limit, namely they have shown that
\begin{equation*}
\frac{\sum_{f \in \Anq} \Lambda_q(f) \Lambda_q(f+1)}{q^n(q-1)} -1= -\frac{1}{q} + O_n\Big(\frac{1}{q^{3/2}} \Big),
\end{equation*}
see the case $k=0$ of \cite[Thm.~1.3]{keating2016}\footnote{Their result estimates $\sum_{f \in \Mnq} \sum_{c \in \FF_q^{\times}}\Lambda_q(f) \Lambda_q(f+c)$, which is the same quantity as $\sum_{f \in \Anq} \Lambda_q(f) \Lambda_q(f+1)$ by \eqref{monicnonmonicrel}.}. Finally, we mention the work of Bary-Soroker and Stix, giving precise answers for $n=3$ and $\Delta=1$ \cite{bary2017}. In \S\ref{secprofthmp} we prove the following theorem, which is a corollary of the general results presented in \S\ref{secmainres}.
\begin{theorem}\label{thmp}
Let $n$ be a positive integer. Let $\Delta$ be a squarefree polynomial in $\mathcal{A}_q$ which is either of degree~$\le n-5$ or of degree $n-1$. We have, for $n \ge 5$,
\begin{equation*}
\frac{\sum_{f \in \Anq} \Lambda_q(f) \Lambda_q(f+\Delta)}{q^n(q-1)}  -1= \frac{-1+a_{\Delta,q}}{q} + O_{n}\Big(\frac{1}{q^{3/2}}\Big).
\end{equation*}
If $\Delta \in \FF_q^{\times}$ and $n \ge 4$ then
\begin{equation}\label{eq:lambdashifted}
\frac{\sum_{f \in \Mnq} \Lambda_q(f) \Lambda_q(f+\Delta)}{q^n} -1= O_{n}\Big(\frac{1}{q}\Big).
\end{equation}
\end{theorem}
The first part of Theorem~\ref{thmp} is the first result showing a dependence of the mean value of $\Lambda_q(f) \Lambda_q(f+\Delta)$ on $\Delta$ (namely, on the linear factors of $\Delta$). In particular, it is the first result that gives us a glimpse into lower-order terms of the Hardy-Littlewood constant, see \S\ref{twprdisc} for more details, where we show that $1+(-1+a_{\Delta,q})/q$ is a first-order approximation for the Hardy-Littlewood constant.

The second part of Theorem~\ref{thmp} is the first result that gives the correct error term for the left hand side of \eqref{eq:lambdashifted}, that is, for the shifted correlation of $\Lambda_q$ over \emph{monic} polynomials.

We now discuss the M\"obius function. The mean value of $\mu_q$ is also well known \cite[Eq.~(5)]{carlitz1932}:
\begin{equation*}
\frac{\sum_{f \in \Anq} \mu_q(f)}{q^n(q-1)} = \frac{\sum_{f \in \Mnq} \mu_q(f)}{q^n} = 0
\end{equation*}
for all $n \ge 2$. This again may be considered as an analogue of the Prime Number Theorem. A conjecture of Chowla \cite{chowla1965}, for which Sarnak has found deep interpretations \cite{sarnak2016}, asserts in particular that
\begin{equation*}
\lim_{x \to \infty}\frac{\sum_{n \le x} \mu(n)\mu(n+h)}{x}=0
\end{equation*}
for all $h \ge 1$. Here $\mu$ is the usual M\"obius function, defined on the positive integers. In the function field setting, Rudnick and Carmon \cite{carmon2014} and Carmon \cite{carmon2015} used an algebro-geometric argument to show that for any $\Delta \in \FF_q[T]\setminus \{0\}$ of degree $<n$, we have 
\begin{equation*}
\frac{\sum_{f \in \Mnq} \mu_q(f) \mu_q(f+\Delta)}{q^n} = O\Big(\frac{n^2}{\sqrt{q}}\Big),
\end{equation*}
from which it follows by the first part of Lemma~\ref{monicnonmoniclem} that
\begin{equation}\label{chowcorm}
\frac{\sum_{f \in \Anq} \mu_q(f) \mu_q(f+\Delta)}{q^n(q-1)} = O\Big(\frac{n^2}{\sqrt{q}}\Big)
\end{equation}
as well. This result is a large-$q$ analogue of Chowla Conjecture. When $\Delta=1$, Keating and Roditty-Gershon have improved \eqref{chowcorm} in the large-$q$ limit, namely they have shown that \cite[Thm.~4.4]{keating2016}
\begin{equation*}
\frac{\sum_{f \in \Anq} \mu_q(f) \mu_q(f+1)}{q^n(q-1)}  =  O_n\Big(\frac{1}{q^{3/2}} \Big).
\end{equation*}
For a recent breakthrough on the function field Chowla Conjecture in the large-$n$ limit with fixed $q$, see the results of the second author and Shusterman \cite{sawin2019}, obtained while this article was in press.

In \S\ref{secprofthmmu} we prove the following theorem, which is again a corollary of the general results presented in \S\ref{secmainres}.
\begin{theorem}\label{thmmu}
Let $n$ be a positive integer. Let $\Delta$ be a squarefree polynomial in $\mathcal{A}_q$ which is either of degree~$\le n-5$ or of degree $n-1$. We have, for $n \ge 5$,
\begin{equation}\label{qsavingmu}
\frac{\sum_{f \in \Anq} \mu_q(f) \mu_q(f+\Delta)}{q^n(q-1)} =  O_{n}\Big(\frac{1}{q^{3/2}}\Big).
\end{equation}
If $\Delta \in \FF_q^{\times}$ and $n \ge 4$ then
\begin{equation}\label{qsavingmonic}
\frac{\sum_{f \in \Mnq} \mu_q(f) \mu_q(f+\Delta)}{q^n} =  O_{n}\Big(\frac{1}{q}\Big).
\end{equation}
\end{theorem}
For $\Delta \neq 1$, \eqref{qsavingmu} gives an additional saving of $q$ compared to previous results. The estimate \eqref{qsavingmonic} is the first estimate which give a saving of $q$ when the average is over \emph{monic} polynomials.

In Theorems~\ref{thmp} and \ref{thmmu}, the condition that $\Delta$ is squarefree comes only in one part of the proof, where we use Theorem~\ref{thmequigenarith}, an equidistribution result which currently requires squarefree-ness. This should not be a fundamental condition, and the general case is expected to be true, although challenging. The same goes also to the range of the degree of $\Delta$, for which the current results do not allow the values $n-4$, $n-3$ and $n-2$.
\subsection{Arithmetic functions on \texorpdfstring{$\FF_q[T]$}{FqT} and previous work}
An arithmetic function on $\FF_q[T]$ is any function $\alpha \colon \FF_q[T] \setminus \{0\} \to \mathbb{C}$.\footnote{Sometimes we use the domain $\FF_q[T]$ instead of $\FF_q[T] \setminus \{0\}$, but we shall never use the value of $\alpha$ at $0$.} If $f \in \FF_q[T] \setminus \{0\}$ has prime factorization $c \cdot \prod_{i=1}^{k} P_i^{e_i}$ where $c \in \FF_q^{\times}$ and $P_i$ distinct primes in $\mathcal{P}_q$, then its \emph{extended factorization type} is the multiset
\begin{equation*}
\lambda_f = \{ (\deg P_i,e_i) : 1 \le i \le k \}.
\end{equation*}
Let $\mathrm{EFT}$ be the set of all extended factorization types.
Following Rodgers \cite[\S2B]{rodgers2018}, we say that an arithmetic function $\beta\colon \FF_q[T] \setminus \{0\}$ is a \emph{factorization function} if the value $\beta(f)$ is determined by $\lambda_f$, i.e. if there is a function $b \colon \mathrm{EFT}\to \mathbb{C}$ such that $\beta(f)= b(\lambda_f)$ for all $f \in \FF_q[T]\setminus \{0\}$.  The function $b$ is not unique, since for instance the value of $b$ on the multiset of $q+1$ $(1,1)$-s can be chosen arbitrarily (as there are only $q$ distinct linear polynomials).

With any function $\alpha \colon \mathrm{EFT} \to \mathbb{C}$ and any prime power $q$, we may associate a factorization function $\alpha_q \colon \FF_q[T]\setminus \{0\} \to \mathbb{C}$ by letting
\begin{equation}\label{alphaqalpha}
\alpha_q(f) = \alpha(\lambda_f).
\end{equation}
An arithmetic function $\beta \colon \FF_q[T]\to \mathbb{C}\setminus \{0\}$ is said to be \emph{even} if $\beta(c\cdot f)=\beta(f)$ for any $c \in \FF_q^{\times}$ and any $f \in \FF_q[T] \setminus \{0\}$. Any factorization function is even. 

From now on we reserve the notation $\alpha_q$ for a factorization function on $\FF_q[T]$ which comes from $\alpha\colon \mathrm{EFT}\to \mathbb{C}$. Although for any specific $q$, $\alpha_q$ does not determine a unique function $\alpha$ such that \eqref{alphaqalpha} holds, we do have a unique $\alpha$ once we look at an infinite number of $q$-s, that is: if $\alpha ( \lambda_f) = \beta(\lambda_f)$ for all $f\in \FF_q[T] \setminus \{0\}$ for infinitely many $q$-s, we must have $\alpha= \beta$. In particular, a family of functions $\{ \alpha_q : q \mbox{ a prime power}\}$ which come from $\alpha$ determines $\alpha$ uniquely.

For functions $\alpha,\beta \colon X \to \mathbb{C}$ and a non-empty finite subset $S \subseteq X$, we denote the mean value of $\alpha$ over $S$ by
\begin{equation*}
\EE_{S} \alpha= \frac{\sum_{f \in S} \alpha(f)}{|S|}
\end{equation*}
and the covariance of $\alpha$ and $\beta$ over $S$ by
\begin{equation*}
\begin{split}
\mathrm{Cov}_{S} \left( \alpha,\beta \right) =\EE_{S} (\alpha \overline{\beta}) - \EE_{S} \alpha \cdot \EE_{S} \overline{\beta} = \EE_{S} ((\alpha-\EE_{S} \alpha)(\overline{\beta}-\EE_{S} \overline{\beta})).
\end{split}
\end{equation*}
Let $\alpha$, $\beta$ be arithmetic functions on $\FF_q[T]$. Many important questions of number theory are encoded in the following covariances:
\begin{equation*}
\mathrm{Cov}_{\mathcal{M}_q}(\alpha,\beta;n,\Delta)= \mathrm{Cov}_{f\in \Mnq} \left( \alpha(f),\beta(f+\Delta) \right)
\end{equation*}
and
\begin{equation*}
\mathrm{Cov}_{\mathcal{A}_q}(\alpha,\beta;n,\Delta) = \mathrm{Cov}_{f\in \Anq} \left( \alpha(f),\beta(f+\Delta) \right),
\end{equation*}
where $n$ is a positive integer and $\Delta$ is a non-zero polynomial of degree $<n$. Let
\begin{equation*}
\mathrm{max} (\alpha;n) = \mathrm{max}_{f \in \Mnq} \left| \alpha(f) \right|.
\end{equation*}
Andrade, Bary-Soroker and Rudnick \cite[Thm.~1.4]{andrade2015} have shown that
\begin{equation}\label{monicest}
\mathrm{Cov}_{\mathcal{M}_q}(\alpha,\beta;n,\Delta) =O_{n,\mathrm{max} (\alpha;n),\mathrm{max} (\beta;n)} \Big(\frac{1}{\sqrt{q}} \Big)
\end{equation}
for any pair of factorization functions $\alpha,\beta$. From \eqref{monicest} and the first part of Lemma~\ref{monicnonmoniclem} we obtain that 
\begin{equation}\label{monicest2}
\mathrm{Cov}_{\mathcal{A}_q}(\alpha,\beta;n,\Delta) =O_{n,\mathrm{max} (\alpha;n),\mathrm{max} (\beta;n)}\Big(\frac{1}{\sqrt{q}} \Big)
\end{equation}
holds as well. Estimate \eqref{monicest} extends the results of Pollack, Bender and Pollack, Bary-Soroker, Carmon and Rudnick, and Carmon concerning the shifted correlation of $\Lambda_q$ and $\mu_q$. We remark that by applying the methods of Pollack \cite{pollack2008polynomial} carefully for general factorization functions (by borrowing the combinatorial ideas in Rodgers \cite{rodgers2018}), one can in fact obtain
\begin{equation*}
\mathrm{Cov}_{\mathcal{A}_q}(\alpha,\beta;n,\Delta) =O_{n,\mathrm{max} (\alpha;n),\mathrm{max} (\beta;n)}\Big(\frac{1}{q} \Big).
\end{equation*}
\subsection{Main results}\label{secmainres}
Our first theorem is a determination of the main term of $\mathrm{Cov}_{\mathcal{A}_q}(\alpha_q,\beta_q;n,\Delta)$ in the limit $q \to \infty$ for most choices of $\Delta$. To state the theorem, we need the notion of Fourier expansion of factorization functions \cite[\S2B]{rodgers2018}, which we now explain.

Let $\alpha_q \colon \FF_q[T] \to \mathbb{C}$ be a family of factorization functions which come from $\alpha \colon \mathrm{EFT}\to\mathbb{C}$ and let $n$ be a positive integer. Let $S_n^{\#}$ be the set of conjugacy classes of $S_n$, identified as usual with partitions $\lambda=(\lambda_1,\ldots,\lambda_k)$ of $n$ ($\lambda_i$ is always non-increasing). We may embed $S_n^{\#}$ in $\mathrm{EFT}$ by identifying $\lambda$ with the multiset 
\begin{equation*}
\{ (\lambda_i, 1): 1 \le i \le k\}.
\end{equation*}
Under this identification, $\alpha |_{S_n^{\#}}$ is a class function on $S_n$. We may expand this function in the basis of irreducible characters of $S_n$, which are also indexed by partitions and we denote them as usual by $\chi_{\lambda}$:
\begin{equation}\label{deffourier}
\alpha |_{S_n^{\#}} (\pi) = \sum_{\lambda \vdash n} \hat{\alpha}_{\lambda} \chi_{\lambda}(\pi) .
\end{equation}
The coefficients $\hat{\alpha}_{\lambda}$ are called the Fourier coefficients of $\alpha$. 
\begin{theorem}\label{maintermcovthm}
Let $\alpha_q$, $\beta_q$ be factorization functions which come from $\alpha,\beta$. Let $n\ge 5$ be an integer and let $\Delta$ be a squarefree polynomial in $\mathcal{A}_q$ which is either of degree~$\le n-5$ or of degree $n-1$. Then
\begin{equation*}
\mathrm{Cov}_{\mathcal{A}_q}(\alpha_q,\beta_q;n,\Delta) =\frac{(a_{\Delta,q}-1) \hat{\alpha}_{(n-1,1)} \overline{\hat{\beta}_{(n-1,1)}} }{q} +O_{n,\mathrm{max} (\alpha;n),\mathrm{max} (\beta;n)}\Big(\frac{1}{q^{3/2}} \Big).
\end{equation*}
\end{theorem}
For an arithmetic interpretation of $\hat{\alpha}_{(n-1,1)}$, see Lemma~\ref{lem:CoeffInterpret}. From Theorem~\ref{maintermcovthm} and Lemma~\ref{monicnonmoniclem} we immediately have the following corollary.
\begin{corollary}\label{cormonic}
Under the assumptions of Theorem~\ref{maintermcovthm}, and the following additional assumptions:
\begin{enumerate}
\item $\Delta$ is of the form $c(T+a)^{k}$ ($c \in \FF_q^{\times}$, $a\in \FF_q$, $k\ge 0$), and
\item $\gcd(q-1,n-\deg \Delta)=1$, or $\mathrm{char}(\FF_q) > 2$, $\alpha=\overline{\beta}$ and $\gcd(\frac{q-1}{2},n-\deg \Delta)=1$,
\end{enumerate}
we have
\begin{equation*}
\mathrm{Cov}_{\mathcal{M}_q}(\alpha_q,\beta_q;n,\Delta) =\frac{(a_{\Delta,q}-1)\hat{\alpha}_{(n-1,1)} \overline{\hat{\beta}_{(n-1,1)}} }{q} +O_{n,\mathrm{max} (\alpha;n),\mathrm{max} (\beta;n)}\Big(\frac{1}{q^{3/2}} \Big).
\end{equation*}
\end{corollary}
In particular, Corollary \ref{cormonic} shows that under additional assumptions, the results of Theorems~\ref{thmp} and \ref{thmmu} on $\mathrm{Cov}_{\mathcal{M}_q}(\Lambda_q,\Lambda_q;n,\Delta)$, $\mathrm{Cov}_{\mathcal{M}_q}(\mu_q,\mu_q;n,\Delta)$ may be extended and improved.
\begin{remark}
In view of the condition in Theorem~\ref{maintermcovthm} that $\Delta$ should be squarefree, $k$ in Corollary \ref{cormonic} is forced to be $0$ or $1$. As mentioned, we expect that the squarefree-ness will be removed eventually. 
\end{remark}
There are three main parts in the proof of Theorem~\ref{maintermcovthm}.

First, we consider the case when $f$ is coprime to $\Delta$. In this case, we take advantage of the averaging over the leading coefficient of $f$ to write (in Proposition~\ref{propcordel}) the sum as a sum over monic polynomials involving $\mathbb F_q$-multiples of the shift $\Delta$, and hence as a covariance of the sums of $\alpha_q$ and $\beta_q$ in short arithmetic progressions of length $q$ with step size $\Delta$. We can detect membership in these progressions by combining Dirichlet characters modulo $\Delta$ with characters ramified at $\infty$, and this allows us to replace this sum with a sum over Dirichlet character in Proposition~\ref{twoidengap}. (We must have $f$ coprime to $\Delta$ to take advantage of Dirichlet characters mod $\Delta$.)

Second we attack this sum over Dirichlet characters. We relate the term corresponding to a Dirichlet character to the $L$-function of that Dirichlet character. When the factorization function is any of the divisor function, the M\"{o}bius function, or the von Mangoldt function, this is a standard manipulation in analytic number theory, but in the general case it requires combinatorial ideas of Rodgers \cite{rodgers2018}. To estimate these sums, expressed in terms of $L$-functions, we need a new equidistribution result (Theorem~\ref{thmequigenarith}) as the existing ones do not quite cover the types of Dirichlet characters we need.

Third, we consider the case when $f$ is not coprime to $\Delta$. The largest contribution comes when $\gcd(f, \Delta)$ is a linear polynomial $L$. We can relate this case to a correlation sum of $f/L$ with shift $\Delta/L$. Because this sum is shorter, we can get by with a weaker estimate, which is provided already by the result \eqref{monicest} of Andrade, Bary-Soroker and Rudnick, which we apply in Lemma~\ref{lemcorbiggcd}. Because their result involves Galois-theoretic methods, our final result involves a combination of these ideas with the $L$-function methods of Pollack, Keating and Roditty-Gershon used in the first two parts of our argument. Both Galois groups and $L$-functions are by this point common approaches to function field analytic number theory problems, but are usually used separately.

%

We conjecture that $\mathrm{Cov}_{\mathcal{M}_q}(\alpha_q,\beta_q;n,\Delta)$ and $\mathrm{Cov}_{\mathcal{A}_q}(\alpha_q,\beta_q;n,\Delta)$ share the same asymptotics:
\begin{conjecture}\label{mainconj}
Let $\alpha_q$, $\beta_q$ be factorization functions which come from $\alpha$, $\beta$. Let $n \ge 5$ be an integer. Let $\Delta$ be a non-zero polynomial of degree $<n$. Then
\begin{equation*}
\mathrm{Cov}_{\mathcal{M}_q}(\alpha_q,\beta_q;n,\Delta) =\frac{(a_{\Delta,q}-1)\hat{\alpha}_{(n-1,1)} \overline{\hat{\beta}_{(n-1,1)}}}{q} +O_{n,\mathrm{max} (\alpha;n),\mathrm{max} (\beta;n)}\Big(\frac{1}{q^{3/2}} \Big).
\end{equation*}
\end{conjecture}
Conjecture~\ref{mainconj} differs from
Corollary~\ref{cormonic} only in that it removes the additional assumptions made in 
the corollary.

Our second result is an improvement of \eqref{monicest} in the case that $\Delta$ is a scalar.
\begin{theorem}\label{thmmn}
Let $\alpha_q$, $\beta_q$ be factorization functions which come from $\alpha$, $\beta$. Let $n \ge 4$. Let $\Delta \in \FF_q^{\times}$. Then
\begin{equation*}
\mathrm{Cov}_{\mathcal{M}_q}(\alpha_q,\beta_q;n,\Delta) =O_{n,\mathrm{max} (\alpha ; n),\mathrm{max} (\beta ; n)}\Big(\frac{1}{q} \Big).
\end{equation*}
\end{theorem}
The main ideas of the proof of Theorem~\ref{thmmn} are completely new. We introduce an $L$-function formula for the correlation of general arithmetic functions, which relates an average over polynomials to an average over certain Dirichlet characters (Proposition~\ref{lemcorexp2}). This falls into the general framework in analytic number theory where we replace an identity, in this case $f_2=f_1+\Delta$, with an average over characters. This case may be surprising because we are detecting an additive identity using multiplicative Dirichlet characters. However, using Dirichlet characters ramified at primes dividing $\Delta$ and at $\infty$, it is possible to do this. The contribution of a given character is closely related to the Dirichlet $L$-function of that character. 

When $\Delta$ is a scalar, we can compose a character with a ring automorphism of $\FF_q[T]$ to get a new character ramified at the same points, which will have the same Dirichlet $L$-function. This gives an additional symmetry of the average over characters (Proposition~\ref{propsym}) which we are able to use in order to get improved asymptotics by first summing over compositions of a given character using elementary Gauss sum estimates, getting some cancellation, and then getting additional cancellation by summing over all characters using $L$-function equidistribution results.

Both Theorem~\ref{maintermcovthm} and Theorem~\ref{thmmn} rely on new equidistribution results, discussed in Appendix~\ref{app:secequi}.

\subsection{Further results}
For an integer $k \ge 2$, the $k$-th divisor function $d_{k,q} \colon \FF_q[T] \to \mathbb{C}$ is defined on monics by
\begin{equation*}
d_{k,q}(f) = \# \{ (f_1, f_2, \ldots, f_k ) : f_1 f_2 \cdots f_k = f , f_i \in \mathcal{M}_q \}.
\end{equation*}
We extend $d_{k,q}$ to non-monics by $d_{k,q}(c\cdot f)=d_{k,q}(f)$ for all $c \in \FF_q^{\times}, f \in \mathcal{M}_q$. The mean value of $d_{k,q}$ is given by \cite[Lem.~2.1]{andrade2015}
\begin{equation*}
\EE_{\Mnq} d_{k,q}= \EE_{\Anq} d_{k,q} = \binom{n+k-1}{n}.
\end{equation*}
The estimate \eqref{monicest}, which was in fact motivated by the `shifted divisor problem', implies that
\begin{equation}\label{eq:shiftedres}
\mathrm{Cov}_{\mathcal{M}_q}(d_{k,q},d_{l,q};n,\Delta) = O_n\Big(\frac{1}{\sqrt{q}} \Big),
\end{equation}
which, by Lemma~\ref{monicnonmoniclem}, implies that
\begin{equation*}
\mathrm{Cov}_{\mathcal{A}_q}(d_{k,q},d_{l,q};n,\Delta) = O_n\Big(\frac{1}{\sqrt{q}} \Big)
\end{equation*}
as well. For $k=l=2$, Keating and Roditty-Gershon have improved \eqref{eq:shiftedres} for $\Delta=1$ \cite[Thms.~4.2]{keating2016}:
\begin{equation*}
\mathrm{Cov}_{\mathcal{A}_q}(d_{2,q},d_{2,q};n,1) = -\frac{(n-1)^2}{q} +O_n\Big(\frac{1}{q^{3/2}} \Big).
\end{equation*}
In \S\ref{secprofthmd} we prove the following theorem, which is again a corollary of the general results presented in \S\ref{secmainres}. In particular, we obtain the main term of $\mathrm{Cov}_{\mathcal{A}_q}(d_{k,q},d_{l,q};n,\Delta)$ for most $\Delta$-s, which turns out to be an interesting combinatorial expression.
\begin{theorem}\label{thmd}
Let $n$ be a positive integer. Let $\Delta$ be a squarefree polynomial in $\mathcal{A}_q$ which is either of degree~$\le n-5$ or of degree $n-1$. For any $k,l\ge 2$ and $n \ge 5$ we have
\begin{equation*}
\mathrm{Cov}_{\mathcal{A}_q}(d_{k,q},d_{l,q};n,\Delta) =\frac{(a_{\Delta,q}-1)(n-1)^2\binom{n+k-2}{n}\binom{n+l-2}{n} }{q} + O_{n,k,l}\Big(\frac{1}{q^{3/2}}\Big).
\end{equation*}
If $\Delta \in \FF_q^{\times}$ and $n \ge 4$ then
\begin{equation*}
\mathrm{Cov}_{\mathcal{M}_q}(d_{k,q},d_{l,q};n,\Delta)= O_{n,k,l}\Big(\frac{1}{q}\Big).
\end{equation*}
\end{theorem}
In the setting of integers, the asymptotics of $(\sum_{n \le x} d_k(n)d_{l}(n+h))/x$ as $x \to \infty$ are known only in the case $k=l=2$, which is due to Ingham \cite{ingham1927}, and $k=2$, $l>2$, which is due to Linnik \cite[Ch.~3]{linnik1963}. If $k$ and $l$ are both greater than $2$, there are complicated conjectures for the asymptotics, which are due to Ivi\'{c} \cite{ivic1997} and Conrey and Gonek \cite{conrey2001}. Theorem~\ref{thmd} can be interpreted as recovering first-order approximation for the arithmetic constants in these conjectures.
 
In \S\ref{secthmsumcoeff} we prove the following.
\begin{theorem}\label{corolvar}
Let $\alpha_q$, $\beta_q$ be factorization functions which come from $\alpha,\beta$. Let $n \ge 5$ be an integer. Let $h$ be an integer such that $0 \le h \le n-5$. Then
\begin{equation*}
\sum_{\Delta \in \mathcal{A}_{h,q}}  \mathrm{Cov}_{\mathcal{M}_q}(\alpha_q,\beta_q;n,\Delta)  = -\sum_{\lambda \vdash n, \lambda_1 = n-h-1} \hat{\alpha}_{\lambda} \overline{ \hat{\beta}_{\lambda}}+ O_{n,\mathrm{max} (\alpha ; n),\mathrm{max} (\beta ; n)}\Big( \frac{1}{\sqrt{q}} \Big).
\end{equation*}
\end{theorem}
Several special cases of Theorem~\ref{corolvar} were proved in \cite{keating2016}, namely $\alpha_q=\beta_q=\Lambda_q$, $\alpha_q=\beta_q=d_{2,q}$ and $\alpha_q =\beta_q=\mu_q$:
\begin{equation*}
\begin{split}
\sum_{\Delta \in \mathcal{A}_{h,q}}  \mathrm{Cov}_{\mathcal{M}_q}(\Lambda_q,\Lambda_q;n,\Delta) &= -1+O_n\Big(\frac{1}{\sqrt{q}}\Big),\\
\sum_{\Delta \in \mathcal{A}_{h,q}}  \mathrm{Cov}_{\mathcal{M}_q}(d_{2,q},d_{2,q};n,\Delta) &= -(n-2h-1)^2 \cdot \indic_{h \le \frac{n}{2}-1} + O_n \Big(\frac{1}{\sqrt{q}}\Big),\\
\sum_{\Delta \in \mathcal{A}_{h,q}}  \mathrm{Cov}_{\mathcal{M}_q}(\mu_q,\mu_q;n,\Delta) &= O_n\Big( \frac{1}{\sqrt{q}} \Big).
\end{split}
\end{equation*}

\subsection{Discussion}
It is natural to ask what kind of technical improvements are needed in order to obtain additional lower order terms, or better error terms, in our main results. We can obtain, without much effort, additional lower order terms in most of the lemmas and propositions appearing in the proofs of Theorems~\ref{maintermcovthm} and \ref{thmmn}. The two exceptions are the following. Firstly, we do not know how to obtain lower order terms in \eqref{eq:indep}, which is proved using a Chebotarev Density Theorem for function fields. Secondly, we do not know how to obtain lower order terms in the equidistribution results of Appendix~\ref{app:secequi}, which are proved using Deligne's equidistribution theorem. Once one is able to obtain lower order terms in these estimates, one additional change is needed. Instead of working in the general setting of factorization functions, we should restrict to the class of \emph{arithmetic functions of von Mangoldt type}, introduced by Hast and Matei \cite[Def.~4.3]{hast2018}. This class is still quite general -- any factorization function can be written as a sum of a function of von Mangoldt type and a function supported on non-squarefree polynomials \cite[Prop.~4.5]{hast2018}. In addition, the von Mangoldt function, the M\"{o}bius function and the divisor functions are all of von Mangoldt type. These functions have the advantage that there is no error term in \eqref{mtsumgen} for most characters $\chi$, which allows us to make use of lower order terms in the equidistribution results.

\section{Preliminaries}
\subsection{Hayes characters}
Here we review the function field analogue of Dirichlet characters, first introduced by Hayes in the paper \cite{hayes1965} which is based on his thesis. We call these characters ``Hayes characters", or sometimes ``generalized arithmetic progression characters". Unless otherwise stated, the proofs of the statements in this section appear in Hayes' original paper. The main difference between Hayes characters and Dirichlet characters is that in the function field setting we can also consider characters modulo the prime at infinity.
\subsubsection{Equivalence relation}\label{erhayes}
Let $\ell$ be a non-negative integer and $M \in \mathcal{A}_q$. We define an equivalence relation $R_{\ell,M}$ on $\mathcal{M}_q$ by saying that $A \equiv B \bmod R_{\ell,M}$ if and only if $A$ and $B$ have the same first $\ell$ next-to-leading coefficients and $A \equiv B \bmod M$. We adopt throughout the following convention: The $j$-th next-to-leading coefficient of a polynomial $f(T) \in \mathcal{M}_q$ with $j > \deg f$ is considered to be $0$. 
It may be shown that there is a well-defined quotient monoid $\mathcal{M}_q/R_{\ell,M}$, where multiplication is the usual polynomial multiplication. An element of $\mathcal{M}_q$ is invertible modulo $R_{\ell,M}$ if and only if it is coprime to $M$. The units of $\mathcal{M}_q/R_{\ell,M}$ form an abelian group, having as identity element the equivalence class of the polynomial $1$. We denote this unit group by $\left( \mathcal{M}_q / R_{\ell,M}\right)^{\times}$. We note that $\left( \mathcal{M}_q / R_{\ell,M}\right)^{\times}$ is isomorphic to
\begin{equation}\label{structunit}
(1+T\FF_q[T])/(1+T^{\ell+1}\FF_q[T]) \oplus (\FF_q[T] / M \FF_q[T] )^{\times},
\end{equation}
and its size is given by
\begin{equation*}
\left| \left(\mathcal{M}_q / R_{\ell,M}\right)^{\times} \right| = q^{\ell} \phi(M).
\end{equation*}
We note that the relation $R_{\ell,M}$ depends only on $\ell$ and the ideal generated by $M$, that is, $R_{\ell,c \cdot M}$ yields the same relation for any $c \in \FF_q^{\times}$.
\subsubsection{Representative sets}
A set of polynomials in $\mathcal{M}_{q}$ is called a representative set modulo $R_{\ell,M}$ if the set contains one and only one polynomial from each equivalence class of $R_{\ell,M}$. The set $\{ f \in \mathcal{M}_{\ell+\deg M,q} : \gcd(f,M)=1 \}$ is a representative set modulo $R_{\ell,M}$. More generally, if $n \ge \ell+\deg M$, then $\{ f \in \Mnq : \gcd(f,M)=1 \}$ is a disjoint union of $q^{n-\ell-\deg M}$ representative sets modulo $R_{\ell,M}$.
\subsubsection{Characters}\label{secchars}
For every character $\chi$ of the finite abelian group $\left( \mathcal{M}_q / R_{\ell,M}\right)^{\times}$, we define $\chi^{\dagger}$ with domain $\mathcal{M}_q$ as follows: If $A$ is invertible modulo $R_{\ell,M}$ and if $\mathfrak{c}$ is the equivalence class of $A$, then $\chi^{\dagger}(A)=\chi(\mathfrak{c})$; If $A$ is not invertible, then $	\chi^{\dagger}(A)=0$.

The set of functions $\chi^{\dagger}$ defined in this way are called the characters of the relation $R_{\ell,M}$, or sometimes ``characters modulo $R_{\ell,M}$". We shall for notational reasons abuse language somewhat and write $\chi$ instead of $\chi^{\dagger}$ to indicate a character of the relation $R_{\ell,M}$ derived from the character $\chi$ of the group $\left( \mathcal{M}_q / R_{\ell,M}\right)^{\times}$. Thus we write $\chi_0$ for the character of $R_{\ell,M}$ which has the value $1$ when $A$ is invertible and the value $0$ otherwise. We denote by $G(R_{\ell,M})$ the set $\{ \chi^{\dagger} : \chi \in \widehat{\left( \mathcal{M}_q / R_{\ell,M}\right)^{\times}} \}$. If $\chi_{1},\chi_{2} \in G(R_{\ell,M})$, then
\begin{equation}\label{ortho1}
\frac{1}{q^{\ell}\phi(M)} \sum_{F} \chi_1(F) \overline{\chi_2}(F) = \begin{cases} 0 & \text{if }\chi_1 \neq \chi_2 ,\\ 1 & \text{if }\chi_1 = \chi_2 ,\end{cases}
\end{equation}
$F$ running through a representative set modulo $R_{\ell,M}$. In particular, if $n \ge \ell + \deg M$ and $\chi_2 = \chi_0$, we have
\begin{equation}\label{orthouse}
\frac{1}{q^{n-\deg M}\phi(M)} \sum_{F \in \Mnq} \chi(F) = \begin{cases} 0 & \text{if }\chi \neq \chi_0 ,\\ 1 & \text{if }\chi = \chi_0. \end{cases}
\end{equation}
We have also
\begin{equation}\label{ortho2}
\frac{1}{q^{\ell}\phi(M)}\sum_{\chi \in G(R_{\ell,M})} \chi(A)\overline{\chi}(B) = \begin{cases} 1 & \text{if }A\equiv B \bmod R_{\ell,M}, \\ 0 & \text{otherwise}. \end{cases}
\end{equation}
We also call the elements of $G(R_{\ell,M})$ ``generalized arithmetic progression characters", because for any $n \ge \ell$ and $A \in \Mnq$, $\chi \in G(R_{\ell,M})$ is constant on the set
\begin{equation*}
\{ f \in \Mnq : f  \equiv A \bmod R_{\ell,M}\} = \{ f\in \FF_q[T] : f \equiv A \bmod M \} \cap \{ f \in \Mnq : \deg(f-A) < n-\ell \}
\end{equation*}
which is an intersection of an arithmetic progression and a short interval. We set, for future use,
\begin{equation*}
\GAP(n,A;R_{\ell,M}) = \{ f \in \Mnq : f  \equiv A \bmod R_{\ell,M}\}.
\end{equation*}
A character $\chi$ modulo $R_{\ell,M}$ is said to be ``primitive modulo $R_{\ell,M}$", or just ``primitive", if $\chi \notin G(R_{\ell -1,M})$ and if for any proper divisor $Q \mid M$, $\chi$ is not of the form $\chi_0$ times a character from $G(R_{\ell,Q})$.
The number of non-primitive characters in $G(R_{\ell,M})$ is bounded from above by
\begin{equation*}
\sum_{P \mid M,\, P \in \mathcal{P}_q} |G(R_{\ell,M/P})| + \indic_{\ell>0} \cdot |G(R_{\ell-1,M})| \le |G(R_{\ell,M})| \cdot (\sum_{P \mid M} \frac{1}{\phi(P)} + \frac{1}{q}) = O_{\deg M}(\frac{|G(R_{\ell,M})|}{q}).
\end{equation*}
If $\ell=0$ and $\deg M > 0$, we can identify $\FF_q^{\times}$ naturally with a subgroup of $\left(\mathcal{M}_q / R_{\ell,M}\right)^{\times} \cong \left( \FF_q[T]/ M \right)^{\times}$. A character $\chi$ modulo $R_{\ell,M}$ is said to be ``even" if $\chi$ is trivial on $\FF_q^{\times}$, and ``odd" otherwise. When either $\ell>0$ or $\deg M = 0$, we consider all characters modulo $R_{\ell,M}$ to be odd. Thus, the number of even characters in $G(R_{\ell,M})$ is $0$ if $\deg M=0$ or $\ell >0$, and is $\frac{|G(R_{\ell,M})|}{q-1}$ otherwise. In particular,
\begin{equation}\label{eq:number of chars}
\# \{ \chi \in G(R_{\ell,M}) :  \chi \text{ primitive and odd}\} = |G(R_{\ell,M})| (1 + O_{\ell,\deg M}(\frac{1}{q})).
\end{equation}

\subsubsection{Structure of \texorpdfstring{$G(R_{\ell,M})$}{G(Rl,M)}}
Any character $\chi \in G(R_{\ell,M})$ is of the form $\chi_1 \cdot \chi_2$ where $\chi_1 \in G(R_{\ell,1})$ and $\chi_2 \in G(R_{0,M})$. This follows, for instance, from counting considerations. Characters modulo $R_{\ell,1}$ are called ``short interval characters" and characters modulo $R_{0,M}$ are called ``(usual) Dirichlet characters".
\subsubsection{$L$-functions}\label{seclfunc}
Let $\chi \in G(R_{\ell,M})$. The $L$-function of $\chi$ is the following series in $u$:
\begin{equation*}
L(u,\chi) = \sum_{f \in \mathcal{M}_q} \chi(f)u^{\deg f},
\end{equation*}
which also admits the Euler product
\begin{equation}\label{eulerlchi}
L(u,\chi) = \prod_{P \in \mathcal{P}_q} (1-\chi(P)u^{\deg P})^{-1}.
\end{equation}
If $\chi$ is the trivial character $\chi_0$ of $G(R_{\ell,M})$, then
\begin{equation*}
L(u,\chi) = \frac{\prod_{P \mid M} (1-u^{\deg P})}{1-qu}.
\end{equation*}
Otherwise, the orthogonality relation \eqref{orthouse} implies that $L(u,\chi)$ is a polynomial in $u$ of degree at most $\ell + \deg M-1$.

The first one to realize that Weil's proof of the Riemann Hypothesis for Function Fields \cite[Thm.~6,~p.~134]{weil1974} implies the Riemann Hypothesis for the $L$-functions of $\chi \in G(R_{\ell,M})$ was Rhin \cite[Thm.~3]{rhin1972} in his thesis (cf. \cite[Thm.~5.6]{effinger1991} and the discussion following it). Hence we know that if we let $a(\chi)$ count the multiplicity of the root $u=1$ in $L(u,\chi)$, and factor $L(u,\chi)$ as
\begin{equation}\label{ldecomproot}
L(u,\chi) = (1-u)^{a(\chi)}\prod_{i=1}^{\deg L(u,\chi) - a(\chi)} (1-\gamma_i(\chi)u),
\end{equation}
then the $\gamma_i(\chi)$-s are $q$-Weil numbers of weight 1, i.e. they are algebraic numbers such that
\begin{equation}\label{absweight1}
\left| \gamma_i(\chi) \right| = \sqrt{q},
\end{equation}
for all $i$, and \eqref{absweight1} is true for the conjugates of $\gamma_i(\chi)$ as well. It is known (for instance, by the functional equation) that if $\chi$ is primitive modulo $R_{\ell,M}$ then 
\begin{equation*}
\deg L(u,\chi)  = \ell+\deg M -1.
\end{equation*}
If $\chi$ is not trivial, we denote by $\Theta_{\chi}$ the conjugacy class of the matrix $\mathrm{diag}(\frac{\gamma_1(\chi)}{\sqrt{q}},\ldots,\frac{\gamma_{\deg L(u,\chi)-a(\chi)}(\chi)}{\sqrt{q}})$ in the unitary group $\mathrm{U}(\deg L(u,\chi)  -a(\chi))$. We sometimes abuse notation and think of $\Theta_{\chi}$ as a specific matrix. Thus,
\begin{equation*}
L(u,\chi) = (1-u)^{a(\chi)} \det(I-u\sqrt{q}\Theta_{\chi}).
\end{equation*}
Taking the logarithmic derivatives of \eqref{eulerlchi} and \eqref{ldecomproot} and comparing coefficients, we obtain
\begin{equation}\label{vonmangiden}
\sum_{f \in \Mnq} \Lambda_q(f) \chi(f)  =- \mathrm{Tr}(\Theta_{\chi}^n) q^{\frac{n}{2}} - a(\chi),
\end{equation}
for all $\chi_0 \neq \chi \in G(R_{\ell,M})$, from which the bound
\begin{equation}\label{vonmangbound}
\Big| \sum_{f \in \Mnq} \Lambda_q(f) \chi(f) \Big| \le (\ell+\deg M - 1) q^{\frac{n}{2}}
\end{equation}
for all $\chi_0 \neq \chi \in G(R_{\ell,M})$ follows. If $\chi$ is odd and primitive then $a(\chi)=0$ and $\Theta_{\chi} \in \mathrm{U}(\ell+\deg M -1)$. 
\subsection{Relations between \texorpdfstring{$\mathrm{Cov}_{\mathcal{A}_q}(\alpha,\beta;n,\Delta)$ and  $\mathrm{Cov}_{\mathcal{M}_q}(\alpha,\beta;n,\Delta)$}{CovAq(alpha,beta;n,Delta) and CovMq(alpha,beta;n,Delta)}}
\begin{lemma}\label{monicnonmoniclem}
	Let $\alpha, \beta \colon \FF_q[T] \to \mathbb{C}$ be two even arithmetic functions. Let $n$ be a positive integer. Let $\Delta \in \FF_q[T]$ be a non-zero polynomial of degree $<n$.
	\begin{enumerate}
		\item We have
		\begin{equation}\label{monicnonmonicrel}
		\mathrm{Cov}_{\mathcal{A}_q}(\alpha,\beta;n,\Delta) = \frac{\sum_{c \in \FF_q^{\times}} \mathrm{Cov}_{\mathcal{M}_q}(\alpha,\beta;n,c\cdot \Delta) }{q-1}.
		\end{equation}
		\item Let $c_1 \in \FF_q^{\times}$, $c_2 \in \FF_q$. Then
		\begin{equation}\label{nopm1res}
		\mathrm{Cov}_{\mathcal{M}_q}(\alpha,\beta;n,\Delta(T)) =  \mathrm{Cov}_{\mathcal{M}_q}(\alpha,\beta;n,\frac{\Delta(c_1 T + c_2)}{c_1^n} ).
		\end{equation}
		If $\alpha = \overline{\beta}$, we also have
		\begin{equation}\label{pm1res}
		\mathrm{Cov}_{\mathcal{M}_q}(\alpha,\beta;n,\Delta(T)) =  \mathrm{Cov}_{\mathcal{M}_q}(\alpha,\beta;n,-\frac{\Delta(c_1 T + c_2)}{c_1^n} ).
		\end{equation}
		\item Let $\Delta=a(T+b)^k$ ($a \in \FF_q^{\times},b \in \FF_q, 0\le k<n$). If $\gcd(n-k,q-1)=1$ then
		\begin{equation*}
		\mathrm{Cov}_{\mathcal{A}_q}(\alpha,\beta;n,\Delta) =  \mathrm{Cov}_{\mathcal{M}_q}(\alpha,\beta;n,\Delta) .
		\end{equation*}
		The same conclusion holds if the following conditions hold simultaneously: $\mathrm{char}(\FF_q) > 2, \alpha = \overline{\beta}, \gcd(n-k, \frac{q-1}{2})=1$.
	\end{enumerate}
\end{lemma}
\begin{proof}
	\begin{enumerate}
		\item We rewrite the right hand side of \eqref{monicnonmonicrel} as follows:
		\begin{equation*}
		\begin{split}
		\frac{\sum_{c \in \FF_q^{\times}} \mathrm{Cov}_{\mathcal{M}_q}(\alpha,\beta;n,c\cdot \Delta) }{q-1} &= \frac{\sum_{c \in \FF_q^{\times}} \sum_{f \in \Mnq} \alpha(f) \overline{\beta}(f+c\cdot \Delta)  }{(q-1)q^n} - \frac{\sum_{c \in \FF_q^{\times}} \EE_{\Mnq} \alpha \cdot  \EE_{\Mnq} \overline{\beta} }{q-1} \\
		&= \frac{\sum_{c \in \FF_q^{\times}} \sum_{f \in \Mnq} \alpha(\frac{f}{c}) \overline{\beta}(\frac{f}{c}+\Delta)  }{(q-1)q^n} - \EE_{\Mnq} \alpha \cdot  \EE_{\Mnq} \overline{\beta} \\
		&= \frac{\sum_{g \in \Anq} \alpha(g) \overline{\beta}(g+\Delta)  }{(q-1)q^n} - \EE_{\Anq} \alpha \cdot  \EE_{\Anq} \overline{\beta}\\
		&= \mathrm{Cov}_{\mathcal{A}_q}(\alpha,\beta;n,\Delta).
		\end{split}
		\end{equation*}
		\item To show that \eqref{nopm1res} holds it suffices to show that
		\begin{equation}\label{permapp}
		\sum_{f \in \Mnq} \alpha(f) \overline{\beta}(f+\Delta) = \sum_{f \in \Mnq} \alpha(f) \overline{\beta}(f+\frac{\Delta(c_1 T +c_2)}{c_1^n}).
		\end{equation}
		Note that $g(T) \mapsto g(c_1 T+c_2)/(c_1^n)$ is a permutation of $\Mnq$ (indeed, its inverse is given by $g(T) \mapsto c_1^n g(\frac{T-c_2}{c_1})$). Applying this permutation to the left hand side of \eqref{permapp}, we get the right hand side of \eqref{permapp}. If $\alpha = \overline{\beta}$, we also have
		\begin{equation*}
		\begin{split}
		\sum_{f \in \Mnq} \alpha(f) \overline{\beta}(f+\frac{\Delta(c_1 T +c_2)}{c_1^n})&= \sum_{f \in \Mnq} \alpha(f+\frac{\Delta(c_1 T +c_2)}{c_1^n}) \overline{\beta}(f)\\
		&= \sum_{f \in \Mnq} \alpha(f) \overline{\beta}(f-\frac{\Delta(c_1 T +c_2)}{c_1^n}),
		\end{split}
		\end{equation*}
		which establishes \eqref{pm1res}.
		\item First assume that $\gcd(n-k,q-1)=1$, a condition which ensures that $c \mapsto c^{n-k}$ is a permutation on $\FF_q^{\times}$. From the two previous parts of the lemma, it suffices to show that for any $c \in \FF_q^{\times}$ we have $c_1 \in \FF_q^{\times}$ and $c_2 \in \FF_q$ such that
		\begin{equation*}
		a(T+b)^k = \frac{c \cdot (a(c_1 T + c_2+b)^k)}{c_1^n}.
		\end{equation*}
		We may take $c_1$ such that
		\begin{equation*}
		c_1^{n-k} = c
		\end{equation*}
		and
		\begin{equation*}
		c_2 = b(c_1-1).
		\end{equation*}
		We now assume instead that $\mathrm{char}(\FF_q) > 2$, $\alpha = \overline{\beta}$ and $\gcd(n-k, \frac{q-1}{2})=1$. The condition $\gcd(n-k, \frac{q-1}{2})=1$ ensures that the subgroup of $\FF_q^{\times}$ generated by $\{ -1, g^{n-k} \}$ (where $g$ is a generator of $\FF_q^{\times}$) is all of $\FF_q^{\times}$. Indeed, $-1=g^{\frac{q-1}{2}}$ and so the subgroup is in fact generated by $g^{\gcd(n-k, \frac{q-1}{2})}$, which is itself a generator if and only if $\gcd(n-k, \frac{q-1}{2})=1$.
		
		From the two previous parts of the lemma, it suffices to show that for any $c \in \FF_q^{\times}$ we have $c_1 \in \FF_q^{\times}$, $c_2 \in \FF_q$ and $\varepsilon \in \{ \pm 1\}$ such that
		\begin{equation*}
		a(T+b)^k  = \varepsilon\frac{c \cdot (a(c_1 T + c_2+b)^k)}{c_1^n}.
		\end{equation*}
		We may take $c_1$ and $\varepsilon$ such that
		\begin{equation*}
		c_1^{n-k} = \varepsilon c
		\end{equation*}
		and
		\begin{equation*}
		c_2 = b(c_1-1).
		\end{equation*}
		
	\end{enumerate}
\end{proof}
\subsection{Some Fourier expansions}
The Fourier expansions of various arithmetic functions were calculated by Rodgers \cite[\S9]{rodgers2018}. 

\begin{proposition}[Rodgers]\label{rodgersexp}
	Let $\mu, \Lambda, d_k \colon \mathrm{EFT} \to \mathbb{C}$ be the functions with which $\mu_q, \Lambda_q, d_{k,q}$ are associated, respectively. Let $n \ge 2$ be an integer and assume that $k \ge 2$.
	\begin{enumerate}
		\item  The Fourier coefficients of $\mu |_{S_n^{\#}}$, defined in \eqref{deffourier}, are given by
		\begin{equation*}
		\hat{\mu}_{\lambda} = \begin{cases}  (-1)^n & \lambda=(1^n), \\ 0 & \text{otherwise.} \end{cases}
		\end{equation*}
		\item The Fourier coefficients of $\Lambda |_{S_n^{\#}}$, defined in \eqref{deffourier}, are given by
		\begin{equation*}
		\hat{\Lambda}_{\lambda} = \begin{cases}  (-1)^{n-r} & \lambda=(r,1^{n-r}) \text{ for some }1 \le r \le n, \\ 0 & \text{otherwise.} \end{cases}
		\end{equation*}
		\item The Fourier coefficients of $d_{k} |_{S_n^{\#}}$, defined in \eqref{deffourier}, are given by
		\begin{equation*}
		(\hat{d_k})_{\lambda} = \begin{cases}  s_{\lambda}(\underbrace{1,\ldots,1}_{k}) & \ell(\lambda) \le k, \\ 0 & \text{otherwise,} \end{cases}
		\end{equation*}
		where $s_{\lambda}$ is the usual Schur function and $\ell(\lambda)$ is the number of parts in $\lambda$. Moreover, if  $\lambda=(\lambda_1,\ldots,\lambda_r)$ with $r \le k$, then
		\begin{equation*}
		s_{\lambda}(\underbrace{1,\ldots,1}_{k}) = \prod_{1 \le i <j \le k} \frac{\lambda_i-\lambda_j+j-i}{j-i},
		\end{equation*}
		with the convention that $\lambda_{r+1}=\lambda_{r+2}=\ldots=\lambda_k=0$.
	\end{enumerate}
\end{proposition}
\begin{corollary}\label{corhighercoeff}
	Let $n \ge 3$ and let $\lambda=(n-1,1)$ be a partition of $n$. In the notation of Proposition~\ref{rodgersexp}, we have
	\begin{equation*}
	\hat{\mu}_{\lambda} = 0, \, \hat{\Lambda}_{\lambda} = -1, \, (\hat{d_k})_{\lambda} = \binom{n+k-2}{k-2}(n-1).
	\end{equation*}
\end{corollary}

\section{Proof of Theorem~\ref{maintermcovthm}}
\subsection{Identities}
For any arithmetic function $\alpha$ on $\FF_q[T]$, any Hayes character $\chi$ and any positive integer $n$, we set 
\begin{equation}\label{defsn}
S(n,\alpha,\chi) = \sum_{f \in \Mnq} \alpha(f) \chi(f).
\end{equation}
For any positive integer $n$ and any non-zero polynomial $\Delta$ of degree $<n$, we set
\begin{equation}\label{defmnd}
\mathcal{M}_{n,q,\Delta} = \{ f \in \Mnq : \gcd(f,\Delta)=1 \}.
\end{equation}
\begin{proposition}\label{twoidengap}
Let $n$ be a positive integer. Let $\Delta \in \mathcal{A}_q$ be a polynomial of degree $<n$. Let $h$ be an integer such that $n \ge h \ge \deg \Delta$. Let $\alpha,\beta \colon \FF_q[T] \to \mathbb{C}$ be two arithmetic functions. Define
\begin{equation*}
\begin{split}
\mathrm{Cov}_{\mathcal{M}_q}(\alpha,\beta;n,\Delta,n-h) &= \mathrm{Cov}_{ A \in \left( \mathcal{M}_q / R_{n-h,\Delta} \right)^{\times}} \Big( \sum_{\substack{f \equiv A \bmod R_{n-h,\Delta},\\ f \in \Mnq}} \alpha(f), \sum_{\substack{f \equiv A \bmod R_{n-h,\Delta}, \\f \in \Mnq}} \beta(f)\Big) .
\end{split}
\end{equation*}
Then the following identities hold.
\begin{enumerate}
\item We have
\begin{equation}\label{covchi}
\mathrm{Cov}_{\mathcal{M}_q}(\alpha,\beta;n,\Delta,n-h) = \left(\frac{1}{q^{n-h} \phi(\Delta)} \right)^2 \sum_{\chi_0 \neq \chi \bmod R_{n-h,\Delta}} S(n,\alpha,\chi) \overline{S(n,\beta, \chi)}.
\end{equation}
\item Assume further that $h \ge \deg \Delta + 1$. Then
\begin{equation}\label{covdiff}
\begin{split}
&\frac{\mathrm{Cov}_{\mathcal{M}_q}(\alpha,\beta;n,\Delta,n-h)}{q^{h-\deg \Delta}} - \frac{\mathrm{Cov}_{\mathcal{M}_q}(\alpha,\beta;n,\Delta,n-h+1)}{q^{h-\deg \Delta - 1}} \\
&\quad = \sum_{\delta \in \mathcal{A}_{h-\deg \Delta -1,q}}  \mathrm{Cov}_{f \in \mathcal{M}_{n,q,\Delta}} (\alpha(f), \beta(f+\delta \Delta))  .
\end{split}
\end{equation}
\end{enumerate}
\end{proposition}

\begin{proof}
\begin{enumerate}
\item For any $A \in \left( \mathcal{M}_q / R_{n-h,\Delta} \right)^{\times}$, \eqref{ortho2} implies that
\begin{equation}\label{indicatorform}
 \indic_{g \in \GAP(n,A;R_{n-h,\Delta}) }= \frac{\sum_{\chi \in G(R_{n-h,\Delta})} \overline{\chi}(A) \chi(g)}{q^{n-h}\phi(\Delta)}
\end{equation}
for all $g \in \Mnq$. From \eqref{indicatorform} we obtain
\begin{equation}\label{orderchangelemma}
\begin{split}
\sum_{g \in \GAP(n,A;R_{n-h,\Delta})} \alpha(g) &= \sum_{g \in \Mnq} \alpha(g) \cdot \indic_{g \in \GAP(n,A;R_{n-h,\Delta})} = \frac{1}{q^{n-h}\phi(\Delta)}\sum_{g \in \Mnq} \alpha(g) \Big(\sum_{\chi \in G(R_{n-h,\Delta})} \overline{\chi}(A)\chi(g) \Big)\\
&= \frac{1}{q^{n-h}\phi(\Delta)} \sum_{\chi \in G(R_{n-h,\Delta})} \overline{\chi}(A) \Big( \sum_{g \in \Mnq} \alpha(g) \chi(g) \Big) =\frac{\sum_{\chi \in G(R_{n-h,\Delta})} \overline{\chi}(A) S(n,\alpha ,\chi )}{q^{n-h}\phi(\Delta)}.
\end{split}
\end{equation}
The term corresponding to $\chi = \chi_0$ in \eqref{orderchangelemma} is
\begin{equation}\label{meanval}
\frac{S(n,\alpha , \chi_0)}{q^{n-h}\phi(\Delta)}= q^{h-\deg \Delta} \EE_{\mathcal{M}_{n,q,\Delta}} \alpha.
\end{equation}
From \eqref{orderchangelemma} and \eqref{meanval}, we have
\begin{equation}\label{varproof}
\begin{split}
\mathrm{Cov}_{\mathcal{M}_q}(\alpha,\beta;n,\Delta,n-h) &=\frac{1}{q^{n-h} \phi(\Delta)} \sum_{A \in \left( \mathcal{M}_q / R_{n-h,\Delta} \right)^{\times}} \Big( \sum_{f \in \GAP(n,A;R_{n-h,\Delta})}\alpha(f) - q^{h-\deg Q} \EE_{\mathcal{M}_{n,q,\Delta}}\alpha   \Big)\\
&\qquad \cdot  \Big( \sum_{f \in \GAP(n,A;R_{n-h,\Delta})} \overline{ \beta}(f) - q^{h-\deg Q} \EE_{\mathcal{M}_{n,q,\Delta}}\overline{\beta}  \Big)\\
&= \frac{1}{q^{n-h} \phi(\Delta)} \sum_{A \in \left( \mathcal{M}_q / R_{n-h,\Delta} \right)^{\times}} \frac{\sum_{\chi_0 \neq \chi_1 \in G(R_{n-h,\Delta}) } \overline{\chi_1}(A) S(n,\alpha,\chi_1)}{q^{n-h}\phi(\Delta)}\\
& \qquad \cdot \frac{\sum_{\chi_0 \neq \chi_2 \in G(R_{n-h,\Delta}) } \chi_2(A) \overline{ S(n,\beta,\chi_2)}}{q^{n-h}\phi(\Delta)}\\
&=\left(\frac{1}{q^{n-h} \phi(\Delta)} \right)^2 \sum_{\substack{\chi_1,\chi_2 \in \\  G(R_{n-h,\Delta}) \setminus \{ \chi_0 \} }} S(n,\alpha,\chi_1) \overline{S(n,\beta,\chi_2)} \frac{\sum_{A \in \left( \mathcal{M}_q / R_{n-h,\Delta}\right)^{\times}} \overline{\chi_1}(A) \chi_2(A)}{q^{n-h} \phi(\Delta)}.
\end{split}
\end{equation}
We conclude the proof of \eqref{covchi} by applying the orthogonality relation \eqref{ortho1} to the right hand side of \eqref{varproof}.
\item By expanding the definition of $\mathrm{Cov}_{\mathcal{M}_q}(\alpha,\beta;n,\Delta,n-h)$, we have
\begin{equation}\label{expandcov}
\begin{split}
\frac{\mathrm{Cov}_{\mathcal{M}_q}(\alpha,\beta;n,\Delta,n-h)}{q^{h-\deg \Delta}} &= 
\frac{1}{q^{n-\deg \Delta}\phi(\Delta)}\sum_{A \in \left( \mathcal{M}_q/R_{n-h,\Delta}\right)^{\times}} \sum_{f ,g \in \GAP(n,A;R_{n-h,\Delta})} \alpha(f) \overline{\beta(g)} \\
&\qquad - q^{h-\deg \Delta} \EE_{\mathcal{M}_{n,q,\Delta}} \alpha \cdot \EE_{\mathcal{M}_{n,q,\Delta}} \overline{\beta}.
\end{split}
\end{equation}
Instead of summing over pairs of polynomials $f,g \in \GAP(n,A;R_{n-h,\Delta})$ in \eqref{expandcov}, we sum over $f \in \GAP(n,A;R_{n-h,\Delta})$ and $\delta :=\frac{g-f}{\Delta}$, an arbitrary polynomial of degree $\le h-\deg \Delta-1$:
\begin{equation}\label{deltasum}
\begin{split}
\frac{\mathrm{Cov}_{\mathcal{M}_q}(\alpha,\beta;n,\Delta,n-h)}{q^{h-\deg \Delta}} &= 
\frac{1}{q^{n-\deg \Delta}\phi(\Delta)}\sum_{A \in \left( \mathcal{M}_q/R_{n-h,\Delta}\right)^{\times}} \sum_{f\in \GAP(n,A;R_{n-h,\Delta})} \sum_{\delta: \deg \delta \le h-\deg \Delta -1} \alpha(f) \overline{\beta}(f+\delta \Delta)\\
& \qquad - q^{h-\deg \Delta}\EE_{\mathcal{M}_{n,q,\Delta}} \alpha \cdot \EE_{\mathcal{M}_{n,q,\Delta}} \overline{\beta} \\
&= \frac{1}{q^{n-\deg \Delta}\phi(\Delta)}\sum_{f \in \mathcal{M}_{n,q,\Delta}} \sum_{\delta: \deg \delta \le h-\deg \Delta -1} \alpha(f) \overline{\beta}(f+\delta \Delta) \\
&\qquad -q^{h-\deg \Delta} \EE_{\mathcal{M}_{n,q,\Delta}} \alpha \cdot \EE_{\mathcal{M}_{n,q,\Delta}} \overline{\beta}.
\end{split}
\end{equation}
Similarly, putting $h-1$ for $h$ in \eqref{deltasum}, we get
\begin{equation}\label{deltasum2}
\begin{split}
\frac{\mathrm{Cov}_{\mathcal{M}_q}(\alpha,\beta;n,\Delta,n-h+1)}{q^{h-1-\deg \Delta}} &= \frac{1}{q^{n-\deg \Delta}\phi(\Delta)}\sum_{f \in \mathcal{M}_{n,q,\Delta}} \sum_{\delta: \deg \delta \le h-\deg \Delta -2} \alpha(f) \overline{\beta}(f+\delta \Delta) \\
&\qquad -q^{h-\deg \Delta-1} \EE_{\mathcal{M}_{n,q,\Delta}} \alpha \cdot \EE_{\mathcal{M}_{n,q,\Delta}} \overline{\beta}.
\end{split}
\end{equation}
(Note: If $h-\deg \Delta - 1 = 0$, then the only polynomial $\delta$ of degree $\le h-\deg \Delta -2$ is the zero polynomial, whose degree is defined to be $-\infty$.) Subtracting \eqref{deltasum2} from \eqref{deltasum}, we get
\begin{equation*}
\begin{split}
&\frac{\mathrm{Cov}_{\mathcal{M}_q}(\alpha,\beta;n,\Delta,n-h)}{q^{h-\deg \Delta}}  - \frac{\mathrm{Cov}_{\mathcal{M}_q}(\alpha,\beta;n,\Delta,n-h+1)}{q^{h-1-\deg \Delta}}  \\
&=\frac{1}{q^{n-\deg \Delta}\phi(\Delta)}\sum_{f \in \mathcal{M}_{n,q,\Delta}} \sum_{\delta: \deg \delta = h-\deg \Delta -1} \alpha(f) \overline{\beta}(f+\delta \Delta) - q^{h-\deg \Delta}(1-\frac{1}{q}) \EE_{\mathcal{M}_{n,q,\Delta}} \alpha \cdot \EE_{\mathcal{M}_{n,q,\Delta}} \overline{\beta}\\
&= \sum_{\delta: \deg \delta = h-\deg \Delta -1} \left( \EE_{f\in \mathcal{M}_{n,q,\Delta}}\alpha(f) \overline{\beta}(f+\delta \Delta) - \EE_{\mathcal{M}_{n,q,\Delta}} \alpha \cdot  \EE_{\mathcal{M}_{n,q,\Delta}} \overline{\beta} \right),
\end{split}
\end{equation*}
which establishes \eqref{covdiff}.
\end{enumerate}
\end{proof}
\subsection{Estimates}
\begin{lemma}\label{lemphi}
Let $\Delta \in \mathcal{A}_q$. We have
\begin{equation*}
\frac{|\Delta|}{\phi(\Delta)} = 1 + \frac{a_{\Delta,q}}{q} + O_{\deg \Delta}\Big( \frac{1}{q^2} \Big).
\end{equation*}
\end{lemma}
\begin{proof}
Since Euler's totient function is multiplicative, we have
\begin{equation*}
\frac{|\Delta|}{\phi(\Delta)} = \prod_{P \mid \Delta, P \in \mathcal{P}_q} \left(1+\frac{1}{|P|-1}\right) = \sum_{D \mid \Delta, D\in \mathcal{M}_q} \frac{\mu^2_q(D)}{\prod_{P \mid D}(|P|-1)} = 1 + \frac{a_{\Delta,q}}{q-1} + \sum_{\substack{D \mid \Delta, D\in \mathcal{M}_q \\ \deg \Delta \ge 2}} \frac{\mu^2_q(D)}{\prod_{P \mid D}(|P|-1)}.
\end{equation*}
In particular, since $|P|-1 \ge |P|/2$ for $P \in \mathcal{P}_q$, 
\begin{equation}\label{eq1phi}
\Big| \frac{|\Delta|}{\phi(\Delta)} - \left( 1 +\frac{a_{\Delta,q}}{q-1} \right) \Big| \le \sum_{\substack{D \mid \Delta, D\in \mathcal{M}_q \\ \deg \Delta \ge 2}} \frac{d_{2,q}(D)}{|D|} \le \frac{d_{2,q}^2(\Delta)}{q^2},
\end{equation}
where in the last inequality we have used $|D| \ge q^2$ and the fact that the number of summands is, by definition, bounded by $d_{2,q}(\Delta)$. If $\Delta=c \cdot \prod_{i} P_i^{e_i}$ is the prime factorization of $\Delta$, then
\begin{equation}\label{eq2phi}
d_{2,q}(\Delta) = \prod_{i} (e_i+1) \le 2^{\sum_{i} e_i} \le 2^{\deg \Delta}.
\end{equation}
The lemma follows from \eqref{eq1phi} and \eqref{eq2phi}.
\end{proof}
Later we sometimes use
\begin{equation}\label{weakphi}
\frac{|\Delta|}{\phi(\Delta)} = 1  + O_{\deg \Delta}\Big( \frac{1}{q} \Big),
\end{equation}
a weaker version of Lemma~\ref{lemphi}.
\begin{lemma}\label{alphasumcoplem}
Let $\alpha\colon \FF_q[T] \to \mathbb{C}$ be a factorization function. Let $n$ be a positive integer. Let $\Delta$ be a non-zero polynomial of degree $<n$. Then
\begin{equation}\label{meancop}
\sum_{f \in \Mnq:\, \gcd(f,\Delta) \neq 1} \alpha(f) = a_{\Delta,q} \sum_{f \in \mathcal{M}_{n-1,q}} \alpha(f\cdot T) + O_{\mathrm{max} (\alpha;n), \deg \Delta }(q^{n-2}).
\end{equation}
\end{lemma}
\begin{proof}
We write 
\begin{equation*}
\sum_{f \in \Mnq:\gcd(f,\Delta) \neq 1} \alpha(f) = S_1 + S_2,
\end{equation*}
where
\begin{equation*}
\begin{split}
S_1 &= \sum_{f \in \Mnq:\, \deg \gcd(f,\Delta) =1} \alpha(f),\\
S_2 &= \sum_{f \in \Mnq:\, \deg\gcd(f,\Delta) \ge 2} \alpha(f).
\end{split}
\end{equation*}
We have
\begin{equation*}
\begin{split}
\left| S_2 \right| &\le \mathrm{max} (\alpha;n)  \cdot \sum_{f \in \Mnq:\, \deg \gcd(f,\Delta) \ge 2} 1 = \mathrm{max} (\alpha;n)  \cdot \sum_{\substack{D \mid \Delta,\, D \in \mathcal{M}_{q}\\ \deg D \ge 2}} \sum_{\substack{f \in \Mnq\\ \gcd(f,\Delta)=D}} 1 \\
&\le \mathrm{max} (\alpha;n)  \cdot \sum_{D \mid \Delta,\, D \in \mathcal{M}_{q},\, \deg D \ge 2} q^{n-\deg D} \le q^{n-2} \cdot \mathrm{max} (\alpha;n)  \cdot d_{2,q}(\Delta).
\end{split}
\end{equation*}
As in the proof of Lemma~\ref{lemphi}, we have $d_{2,q}(\Delta) \le 2^{\deg \Delta}$ and so 
\begin{equation}\label{s2copest}
\left| S_2 \right| = O_{\mathrm{max} (\alpha;n), \deg \Delta}(q^{n-2}).
\end{equation}
Let $X = \{ T-a : a \in \FF_q,\, \Delta(a) = 0\}$. 
By inclusion-exclusion, we have
\begin{equation*}
\begin{split}
S_1 &= \sum_{i=1}^{n} (-1)^{i-1}\sum_{\substack{ \{L_1,\ldots, L_i \} \subseteq X \\ L_j \text{ distinct}}}  \sum_{ \substack{f \in \Mnq\\ \prod_{j=1}^{i} L_j \mid f} }\alpha(f)\\
&= S_3+S_4
\end{split}
\end{equation*}
where
\begin{equation*}
\begin{split}
S_3 &= \sum_{T-a \in X} \sum_{g \in \mathcal{M}_{n-1,q}} \alpha(g(T) \cdot (T-a)),\\
S_4 &= \sum_{i=2}^{n} (-1)^{i-1}\sum_{\substack{ \{L_1,\ldots, L_i \} \subseteq X \\ L_j \text{ distinct}}}  \sum_{ \substack{f \in \Mnq\\ \prod_{j=1}^{i} L_j \mid f} }\alpha(f).
\end{split}
\end{equation*}
We estimate $S_4$ as follows:
\begin{equation}\label{s3copest}
\begin{split}
\left| S_4 \right| &\le \mathrm{max} (\alpha;n) \cdot \sum_{i=2}^{n} \binom{|X|}{i} q^{n-i} \le \mathrm{max} (\alpha;n) \cdot q^{n-2} \cdot 2^{|X|}\\
& \le \mathrm{max} (\alpha;n) \cdot q^{n-2}\cdot 2^{\deg \Delta} = O_{\mathrm{max} (\alpha;n),\deg \Delta}(q^{n-2}).
\end{split}
\end{equation}
Since $\alpha$ is a factorization function, and $g(T)\cdot (T-a), g(T+a) \cdot T$ have the same extended factorization type, we have
\begin{equation}\label{s4copest}
S_3 =a_{\Delta,q} \sum_{f \in \mathcal{M}_{n-1,q}} \alpha(f\cdot T).
\end{equation}
From \eqref{s2copest}, \eqref{s3copest} and \eqref{s4copest} we obtain \eqref{meancop} as needed. 
\end{proof}
The previous two lemmas were elementary. The following lemma requires a deeper result on the cycle structure of polynomials over finite fields.
\begin{lemma}\label{lemcorbiggcd}
Let $\alpha,\beta \colon \FF_q[T] \to \mathbb{C}$ be factorization functions. Let $n$ be a positive integer. Let $\Delta$ be a non-zero polynomial of degree $<n$. Then
\begin{equation}\label{noncopiden}
\sum_{f \in \Mnq:\,\gcd(f,\Delta) \neq 1} \alpha(f) \overline{\beta}(f+\Delta) = a_{\Delta,q} \cdot \frac{\sum_{f \in \mathcal{M}_{n-1,q}}\alpha(f \cdot T) \sum_{g \in \mathcal{M}_{n-1,q}}\overline{\beta}(g \cdot T)}{q^{n-1}} + O_{\mathrm{max} (\beta;n),\mathrm{max} (\alpha;n),n}\Big( q^{n-\frac{3}{2}} \Big).
\end{equation}
\end{lemma}
\begin{proof}
The same reasoning as in the proof of Lemma~\ref{alphasumcoplem} shows that
\begin{equation}\label{gcd1andhigher}
\begin{split}
\sum_{f \in \Mnq:\,\gcd(f,\Delta) \neq 1} \alpha(f) \overline{\beta}(f+\Delta) &= \sum_{a \in \FF_q: \Delta(a)=0} \sum_{g \in \mathcal{M}_{n-1,q}} \alpha((T-a)\cdot g) \overline{\beta}\Big((T-a) \cdot (g + \frac{\Delta}{T-a}) \Big) \\
&\qquad + O_{\mathrm{max} (\alpha;n),\mathrm{max} (\beta;n),\deg \Delta}(q^{n-2}).
\end{split}
\end{equation}
We fix $a \in \FF_q$ such that $\Delta(a)=0$.
Let $\lambda$ be a partition of ${n-1}$. We say that a polynomial $g$ of degree $n-1$ is of type $\lambda$ if it is squarefree and the degrees of its factors coincide with the parts of $\lambda$. For any factorization function $\gamma$, we let $\gamma(\lambda)$ be the common value of $\gamma$ on polynomials of type $\lambda$ (if there is no polynomial of type $\lambda$ in $\mathcal{M}_{n-1,q}$, we set $\gamma(\lambda)=0$). 

We introduce two probabilities: $\PP_{S_{n-1}}(\lambda)$ is the probability that a uniformly chosen element in $S_{n-1}$ has cycle structure given by $\lambda$, while $\PP_{\mathcal{M}_{n-1,q}}(\lambda)$ is the probability that a uniformly chosen element in $\mathcal{M}_{n-1,q}$ has type $\lambda$. It is well known that \cite[Lem.~2.1]{andrade2015}
\begin{equation}\label{equilamb}
\PP_{\mathcal{M}_{n-1,q}}(\lambda) = \PP_{S_{n-1}}(\lambda) + O_n\Big(\frac{1}{q}\Big).
\end{equation}
Andrade, Bary-Soroker, and Rudnick \cite[Thm.~1.4]{andrade2015} proved that for any partitions $\lambda_1,\lambda_2$ of $n-1$, we have
\begin{equation}\label{eq:indep}
\PP_{g \in \mathcal{M}_{n-1,q}} (g,g+\frac{\Delta}{T-a} \text{ are of type } \lambda_1,\lambda_2 ) = \PP_{S_{n-1}}(\lambda_1) \PP_{S_{n-1}}(\lambda_2) + O_n\Big(\frac{1}{\sqrt{q}}\Big).
\end{equation}
Denote by $(\lambda_i,1)$ the partition of $n$ obtained by adjoining to the partition $\lambda_i$ a part of size $1$. Let
\begin{equation*}
p_{q,\lambda_1,\lambda_2}:=\PP_{g \in \mathcal{M}_{n-1,q}}((T-a)g,(T-a)(g+\frac{\Delta}{T-a}) \text{ are of type } (\lambda_1,1),(\lambda_2,1) ).
\end{equation*}
We have
\begin{equation*}
p_{q,\lambda_1,\lambda_2} = p_{q,\lambda_1,\lambda_2,1} + p_{q,\lambda_1,\lambda_2,2},
\end{equation*}
where
\begin{equation*}
\begin{split}
p_{q,\lambda_1,\lambda_2,1}&=\PP_{g \in \mathcal{M}_{n-1,q}}((T-a)g,(T-a)(g+\frac{\Delta}{T-a}) \text{ are of type } (\lambda_1,1),(\lambda_2,1), \\
&\qquad \qquad \qquad \text{ and } g\text{ and } g+\frac{\Delta}{T-a}, \text{ are coprime with } T-a ),\\
p_{q,\lambda_1,\lambda_2,2}&=\PP_{g \in \mathcal{M}_{n-1,q}}((T-a)g,(T-a)(g+\frac{\Delta}{T-a}) \text{ are of type } (\lambda_1,1),(\lambda_2,1),\\ &\qquad \qquad \qquad \gcd(T-a, g (g+\frac{\Delta}{T-a})) \neq 1).
\end{split}
\end{equation*}
Since the probability that both $g$, $g+\frac{\Delta}{T-a}$ are coprime to $T-a$ is $1+O_n(\frac{1}{q})$, we have from \eqref{eq:indep}
\begin{equation*}
p_{q,\lambda_1,\lambda_2,2} = O_n\Big(\frac{1}{q}\Big),\quad  p_{q,\lambda_1,\lambda_2,1} = \PP_{S_{n-1}}(\lambda_1) \PP_{S_{n-1}}(\lambda_2) + O_n\Big(\frac{1}{\sqrt{q}}\Big),
\end{equation*}
and so
\begin{equation*}
p_{q,\lambda_1,\lambda_2} = \PP_{S_{n-1}}(\lambda_1) \PP_{S_{n-1}}(\lambda_2) + O_n\Big(\frac{1}{\sqrt{q}}\Big).
\end{equation*}
Since $\{ \PP_{S_{n-1}}(\lambda_1) \PP_{S_{n-1}}(\lambda_2) \}_{\lambda_1,\lambda_2 \vdash n-1}$ sum to $1$, the probabilities $p_{q,\lambda_1,\lambda_2}$ sum to $1+O_n\Big(\frac{1}{\sqrt{q}}\Big)$, and we obtain
\begin{equation}\label{ave2n1}
\begin{split}
\frac{1}{q^{n-1}}\sum_{g \in \mathcal{M}_{n-1,q}} \alpha((T-a)\cdot g) \overline{\beta}\Big((T-a) \cdot (g + \frac{\Delta}{T-a}) \Big) &= \sum_{\lambda_1,\lambda_2 \vdash n-1} \Big( \PP_{S_{n-1}}(\lambda_1) \PP_{S_{n-1}}(\lambda_2) + O_n(\frac{1}{\sqrt{q}}) \Big)\\
& \qquad \quad  \cdot \alpha((\lambda_1,1))\overline{\beta}((\lambda_2,1)) + O_{n,\mathrm{max} (\alpha;n),\mathrm{max} (\beta;n)}\Big(\frac{1}{\sqrt{q}}\Big) \\
&= \Big( \sum_{\lambda_1 \vdash n-1} \PP_{S_{n-1}}(\lambda_1) \alpha((\lambda_1,1)) \Big) \Big( \sum_{\lambda_2 \vdash n-1} \PP_{S_{n-1}}(\lambda_2) \overline{\beta}((\lambda_2,1)) \Big) \\
& \qquad + O_{n,\mathrm{max} (\alpha;n),\mathrm{max} (\beta;n)}\Big(\frac{1}{\sqrt{q}}\Big).
\end{split} 
\end{equation}
By \eqref{equilamb} and an argument similar to the above we have
\begin{equation}\label{ave1n1}
\frac{1}{q^{n-1}}\sum_{g \in \mathcal{M}_{n-1,q}} \alpha((T-a)\cdot g)  = \sum_{\lambda_1 \vdash n-1} \PP_{S_{n-1}}(\lambda_1) \alpha((\lambda_1,1)) + O_{n, \mathrm{max} (\alpha;n)}\Big(\frac{1}{q}\Big),
\end{equation}
and the same holds with $\beta$ in place of $\alpha$. From \eqref{gcd1andhigher}, \eqref{ave2n1} and \eqref{ave1n1}, we obtain \eqref{noncopiden} as needed.
\end{proof}

\begin{proposition}\label{propcordel}
Let $\alpha,\beta \colon \FF_q[T] \to \mathbb{C}$ be factorization functions. Let $n$ be a positive integer. Let $\Delta$ be a non-zero polynomial of degree $<n$ and $c \in \FF_q^{\times}$. Then
\begin{equation*}
\begin{split}
\sum_{c \in \FF_q^{\times}} \mathrm{Cov}_{f\in \mathcal{M}_{n,q,\Delta}}(\alpha(f), \beta(f+c\Delta)) &= q \cdot \mathrm{Cov}_{\mathcal{A}_q}(\alpha,\beta;n,\Delta) \\
&\qquad -a_{\Delta,q} \Big( \EE_{\Anq} \alpha - \EE_{f \in \mathcal{A}_{n-1,q}} \alpha(f \cdot T) \Big) \\
&\quad \qquad \cdot \Big( 
\EE_{\Anq} \overline{\beta} - \EE_{f \in \mathcal{A}_{n-1,q}} \overline{\beta}(f \cdot T) \Big)\\
&\qquad +  O_{n,\mathrm{max} (\alpha;n),\mathrm{max} (\beta;n)}\left( \frac{1}{\sqrt{q}} \right).
 \end{split} 
\end{equation*}
\end{proposition}
\begin{proof}
Fix $c \in \FF_q^{\times}$. We have, from Lemmas~\ref{lemphi} and \ref{lemcorbiggcd},
\begin{equation}\label{covdel1}
\begin{split}
\EE_{f\in \mathcal{M}_{n,q,\Delta}} \alpha(f)\overline{\beta}(f+c\Delta) &=  \frac{|\Delta|}{\phi(\Delta)} \cdot \frac{\sum_{f\in \Mnq} \alpha(f) \overline{\beta}(f+c\Delta) - \sum_{f\in \Mnq, \gcd(f,\Delta) \neq 1} \alpha(f) \overline{\beta}(f+c\Delta)}{q^n}\\
&= (1+\frac{a_{\Delta,q}}{q} + O_n(q^{-2})) \Big( \EE_{f \in \Mnq} \alpha(f) \overline{\beta}(f+c\Delta)  \\
& \qquad \quad -\frac{a_{\Delta,q}}{q} \EE_{f \in \mathcal{M}_{n-1,q}} \alpha(f \cdot T) \cdot  \EE_{f \in \mathcal{M}_{n-1,q}} \overline{\beta}(f \cdot T) \Big) \\
& \qquad +  O_{\mathrm{max} (\alpha;n),\mathrm{max} (\beta;n),n}\left( q^{-\frac{3}{2}} \right)\\
&=  (1+\frac{a_{\Delta,q}}{q}) \EE_{f \in \Mnq} \alpha(f) \overline{\beta}(f+c\Delta) - \frac{a_{\Delta,q}}{q} \EE_{f \in \mathcal{M}_{n-1,q}} \alpha(f \cdot T) \cdot  \EE_{f \in \mathcal{M}_{n-1,q}} \overline{\beta}(f \cdot T)\\
&\qquad + O_{\mathrm{max} (\alpha;n),\mathrm{max} (\beta;n),n}\left( q^{-\frac{3}{2}} \right).
\end{split}
\end{equation}
Similarly, from Lemmas~\ref{lemphi} and \ref{alphasumcoplem} we have
\begin{equation}\label{covdel2}
\begin{split}
\EE_{\mathcal{M}_{n,q,\Delta}} \alpha  &=  (1+\frac{a_{\Delta,q}}{q})\EE_{\Mnq} \alpha - \frac{a_{\Delta,q}}{q} \EE_{f\in \mathcal{M}_{n-1,q}} \alpha(f \cdot T) +O_{\mathrm{max} (\alpha;n), n }(q^{-2}).
\end{split}
\end{equation}
From \eqref{covdel1} and \eqref{covdel2} we obtain
\begin{equation}\label{covdel3}
\begin{split}
\EE_{f\in \mathcal{M}_{n,q,\Delta}} \alpha(f) \overline{\beta}(f+c\Delta) - \EE_{\mathcal{M}_{n,q,\Delta}} \alpha \cdot \EE_{\mathcal{M}_{n,q,\Delta}} \overline{\beta}  &= (1+\frac{a_{\Delta,q}}{q}) \Big( \EE_{f\in \Mnq} \alpha(f) \overline{\beta}(f+c\Delta) - \EE_{\Mnq} \alpha \cdot \EE_{\Mnq}  \overline{\beta} \Big)\\
&\qquad -\frac{a_{\Delta,q}}{q} \Big( \EE_{\Mnq} \alpha - \EE_{f \in \mathcal{M}_{n-1,q}} \alpha(f \cdot T) \Big) \\
&\quad \qquad \cdot \Big( 
\EE_{\Mnq} \overline{\beta} - \EE_{f \in \mathcal{M}_{n-1,q}} \overline{\beta}(f \cdot T) \Big)\\
&\qquad +  O_{\mathrm{max} (\alpha;n),\mathrm{max} (\beta;n),n}\left( q^{-\frac{3}{2}} \right).
 \end{split} 
\end{equation}
Summing \eqref{covdel3} over $c \in \FF_q^{\times}$ and applying \eqref{monicnonmonicrel}, we have
\begin{equation}\label{covdel4}
\begin{split}
\sum_{c \in \FF_q^{\times}} \mathrm{Cov}_{f\in \mathcal{M}_{n,q,\Delta}}(\alpha(f), \beta(f+c\Delta)) &= (1+\frac{a_{\Delta,q}}{q}) (q-1) \mathrm{Cov}_{\mathcal{A}_q}(\alpha,\beta;n,\Delta) \\
&\qquad -a_{\Delta,q} \Big( \EE_{\Anq} \alpha - \EE_{f \in \mathcal{A}_{n-1,q}} \alpha(f \cdot T) \Big) \\
&\quad \qquad \cdot \Big( 
\EE_{\Anq} \overline{\beta} - \EE_{f \in \mathcal{A}_{n-1,q}} \overline{\beta}(f \cdot T) \Big)\\
&\qquad +  O_{\mathrm{max} (\alpha;n),\mathrm{max} (\beta;n),n}\left( \frac{1}{\sqrt{q}} \right).
 \end{split} 
\end{equation}
From \eqref{covdel4} and \eqref{monicest2} we conclude the proof of the proposition.
\end{proof}

\subsection{Additive decomposition of character sums, after Rodgers}\label{secrod}
The following results are generalizations of results of Rodgers \cite{rodgers2018}. Since the proofs are very similar, we refer to Rodgers' work when appropriate. 
\begin{lemma}\label{lemboundfact}
Let $\alpha \colon \FF_q[T] \to \mathbb{C}$ be a factorization function. Let $\chi$ be a Hayes character in $G(R_{\ell,M}) \setminus \{ \chi_0 \}$. 
\begin{enumerate}
\item We have
\begin{equation}\label{bndrodg1}
|S(n,\alpha,\chi)| =O_{\max(\alpha;n),n,\ell+\deg M}(q^{\frac{n}{2}}).
\end{equation}
\item If $\alpha$ is supported on non-squarefree polynomials and $\chi^2 \neq \chi_0$, then
\begin{equation}\label{bndrodg2}
S(n,\alpha,\chi) = O_{\max(\alpha;n),n,\ell+\deg M}(q^{\frac{n-1}{2}}).
\end{equation}
\end{enumerate}
\end{lemma}
\begin{proof}
In the special case $\ell=0$, $M=T^k$, the bounds \eqref{bndrodg1} and \eqref{bndrodg2} were proved in \cite[Lem.~6.3]{rodgers2018}, \cite[Lem.~6.2]{rodgers2018} respectively. The proofs work for any $\ell$ and $M$. The dependence on $\ell+\deg M$ comes from \eqref{vonmangbound}, which generalizes \cite[Eq.~(26)]{rodgers2018}.
\end{proof}
The next lemma requires some notation and definitions. For any partition $\lambda=(\lambda_1,\ldots,\lambda_k)$ of $n$, let $\chi_{\lambda}$ be the irreducible character of $S_n$ associated with $\lambda$. Let $\chi_{\lambda,q} \colon \FF_q[T] \to \mathbb{C}$ be the factorization function defined as follows (cf. \cite[\S2B]{rodgers2018}):
\begin{equation*}
\chi_{\lambda,q}(f) = \begin{cases} \chi_{\lambda}(\pi_f) & \mbox{if $f$ is squarefree of degree $n$,} \\ 0 & \mbox{otherwise}, \end{cases}
\end{equation*}
where $\pi_f$ is any element of $S_n$ in the conjugacy class associated with $f$. We denote by $s_{\lambda}$ the usual Schur function, and recall that its definition may be extended to unitary matrices, by evaluating $s_{\lambda}$ at the corresponding eigenvalues (cf. \cite[\S5B, \S5D]{rodgers2018}). 
We let $\lambda'$ denote the partition of $n$ conjugate to $\lambda$. 
\begin{lemma}\label{lemschur}
Let $\chi \in G(R_{\ell,M})$ such that $\chi^2 \neq \chi_0$. Let $\lambda \vdash n$. Then
\begin{equation}\label{mtsumschur}
S(n,\chi_{\lambda,q}, \chi) = q^{\frac{n}{2}} (-1)^n s_{\lambda'}(\Theta_{\chi})+O_{n,\ell+\deg M}(q^{\frac{n-1}{2}}).
\end{equation}
Moreover, if $\alpha_q$ is a factorization function which comes from $\alpha$, then
\begin{equation}\label{mtsumgen}
S(n,\alpha_q,\chi) = q^{\frac{n}{2}}\sum_{\lambda \vdash n}  \hat{\alpha}_{\lambda} s_{\lambda'}(\Theta_{\chi})+O_{\max(\alpha_q;n),n,\ell+\deg M}(q^{\frac{n-1}{2}}).
\end{equation}
\end{lemma}
\begin{proof}
In the special case $\ell=0,M=T^k$, the estimate \eqref{mtsumschur} was proven in \cite[Thm.~7.1]{rodgers2018}. The proof works as is for any $\ell$ and $M$.

Let $\mathrm{EFT}_n \subseteq \mathrm{EFT}$  consist of those elements of $\mathrm{EFT}$ that correspond to polynomials of degree $n$. To prove \eqref{mtsumgen}, note that both $S(n,\alpha_q,\chi)$ and $\sum_{\lambda \vdash n}  \hat{\alpha}_{\lambda} s_{\lambda'}(\Theta_{\chi})$ depend linearly on $\alpha|_{\mathrm{EFT}_n}$, and so it suffices to prove \eqref{mtsumgen} for a set of $\alpha$-s whose restriction to $\mathrm{EFT}_n$ spans the vector space of functions $\mathrm{EFT}_n \to \mathbb{C}$. Such a set is given by 
\begin{equation}\label{basisfacts}
\{ \chi_{\lambda} \}_{\lambda \vdash n} \cup \{ \indic_{\nu} : \nu \in \mathrm{EFT}_n, \, \mbox{$\nu$ corresponds to non-squarefree polynomials}\}.
\end{equation}
We immediately verify \eqref{mtsumgen} for these functions using \eqref{mtsumschur} and \eqref{bndrodg2}.
\end{proof}

\begin{lemma}\label{lemquadchars}
The number of character $\chi \in G(R_{\ell,M})$ such that $\chi^2 = \chi_0$ is $O_{\deg M}(1)$ if $q$ is odd, and $O(q^{\lfloor \frac{\ell + 1}{2} \rfloor + \lfloor \frac{\deg M}{2} \rfloor })$ if $q$ is even.
\end{lemma}
\begin{proof}
We start with the case $\ell=0$. From \eqref{structunit}, the 2-torsion in $G(R_{0,M})\cong R_{0,M}$ corresponds to solutions of 
\begin{equation}\label{quadcong}
x^2 \equiv 1 \bmod {M} , \quad \deg x < \deg M.
\end{equation}
If $q$ is odd, each such $x$ defines a unique monic divisor of $M$ via $\gcd(x-1,M)$. By \eqref{eq2phi}, the number of monic divisors of $M=\prod_{i=1}^{m} P_i^{e_i}$ is bounded by $O_{\deg M}(1)$, as needed. If $q$ is even, \eqref{quadcong} becomes $M \mid (x+1)^2$. We may write $M$ as $M=M_0 \cdot M_1^2$ where $M_0$, $M_1$ are coprime and $M_0$ squarefree. The condition $M \mid (x+1)^2$ becomes $M_0 M_1 \mid x+1$, which bounds the size of the 2-torsion by $O(q^{\deg M_1}) = O(q^{\lfloor \frac{\deg M}{2}\rfloor })$.

We now proceed to prove the lemma in the case $M=1$. By \eqref{structunit}, $G(R_{\ell,1}) \cong R_{\ell,1} \cong (1+T\FF_q[T])/(1+T^{\ell+1}\FF_q[T])$. In particular, $R_{\ell,1}$ is of order $q^{\ell}$. If $q$ is odd, the 2-torsion is trivial. If $q$ is even, the 2-torsion corresponds to polynomials of the form
\begin{equation*}
1+T\cdot f, \quad \deg f < \ell,
\end{equation*}
such that
\begin{equation*}
T^{\ell+1} \mid (1+T\cdot f)^2 - 1. 
\end{equation*}
Since
\begin{equation*}
(1+T\cdot f)^2 - 1 = T^2 f^2,
\end{equation*}
we actually count polynomials $f$ of degree $<\ell$ such that $T^{\lceil \frac{\ell-1}{2} \rceil} \mid f$, and their number is $q^{\lfloor \frac{\ell+1}{2} \rfloor}$.

The isomorphism
\begin{equation*}
G(R_{\ell,M}) \cong G(R_{0,M})  \oplus G(R_{\ell,1}).
\end{equation*}
proves the case of general $\ell$ and $M$ from the last two cases.
\end{proof}
We have the following proposition, which relies crucially on the equidistribution result Theorem~\ref{thmequigenarith}.
\begin{proposition}\label{avchiviaschur}
Let $\alpha_q,\beta_q$ be factorization functions which come from $\alpha,\beta$. Let $(\ell,M)\in \mathbb{N} \times \mathcal{A}_q$ where $M$ is squarefree and either $\ell \ge 4$ or $\ell=0$ and $\deg M \ge 2$. Then
\begin{equation}\label{lemcovsum}
\begin{split}
  \frac{\sum_{\chi \in G(R_{\ell,M})\setminus \{ \chi_0\}} q^{-n} S(n,\alpha_q,\chi)\overline{S(n,\beta_q,\chi)} }{|G(R_{\ell,M})|} &= \sum_{\substack{\lambda \vdash n \\ \lambda_1 \le \ell+\deg M -1}} \hat{\alpha}_{\lambda} \overline{\hat{\beta}_{\lambda}} \\
 & \qquad + O_{\max(\alpha_q;n),\max(\beta_q;n),n,\ell,\deg M}\Big(\frac{1}{\sqrt{q}}\Big).
 \end{split}
\end{equation}
\end{proposition}
\begin{proof}
As in the proof of Lemma~\ref{lemschur}, both $\sum_{\chi \in G(R_{\ell,M})\setminus \{ \chi_0\}} S(n,\alpha_q,\chi)\overline{S(n,\beta_q,\chi)}$ and $\sum_{\substack{\lambda \vdash n \\ \lambda_1 \le \ell +\deg M -1}} \hat{\alpha}_{\lambda} \overline{\hat{\beta}_{\lambda}}$ depend linearly on $\alpha|_{\mathrm{EFT}_n}$ and conjugate-linearly on $\beta|_{\mathrm{EFT}_n}$. Hence it suffices to prove \eqref{lemcovsum} for $\alpha,\beta$ in the set \eqref{basisfacts}.

First assume that $\alpha = \indic_{\nu}$ for some $\nu\in\mathrm{EFT}_n$ which corresponds to non-squarefree polynomials. Then $\hat{\alpha}_{\lambda}=0$ for all $\lambda \vdash n$. Thus we need to prove that
\begin{equation}\label{cov1ftb}
 \frac{\sum_{\chi \in G(R_{\ell,M})\setminus \{\chi_0\}}  q^{-n} S(n,\indic_{\nu},\chi) \overline{S(n,\beta,\chi)} }{|G(R_{\ell,M}|} = O_{\max(\beta_q;n),n,\ell,\deg M}\Big(\frac{1}{\sqrt{q}}\Big).
\end{equation}
By breaking the sum over $G(R_{\ell,M}) \setminus \{ \chi_0 \}$ in the left hand side of \eqref{cov1ftb} into two sums, one over $\chi$-s such that $\chi^2 \neq \chi_0$ and another over the rest of the $\chi$-s, and then applying Lemma~\ref{lemboundfact} and Lemma~\ref{lemquadchars}, we obtain the right hand side of \eqref{cov1ftb}. The argument also works if $\beta=\indic_{\nu}$. 

We now assume that  $\alpha=\chi_{\lambda^{(1)}}, \beta=\chi_{\lambda^{(2)}}$ for two partitions $\lambda^{(i)} \vdash n$. In this case we need to prove that
\begin{equation}\label{avchiind}
 \frac{\sum_{\chi \in G(R_{\ell,M})\setminus \{ \chi_0\}}   q^{-n} S(n,\chi_{\lambda^{(1)},q},\chi)\overline{S(n,\chi_{\lambda^{(2)},q},\chi)} }{|G(R_{\ell,M})|} =  \indic_{\substack{\lambda^{(1)} = \lambda^{(2)}\\ \lambda^{(1)}_1 \le \ell+\deg M - 1}} + O_{n,\ell+\deg M}\Big(\frac{1}{\sqrt{q}}\Big).
\end{equation}
We write
\begin{equation}\label{avchis1s2def}
\frac{\sum_{\chi \in G(R_{\ell,M}) \setminus \{\chi_0\}} q^{-n} S(n,\chi_{\lambda^{(1)},q},\chi)\overline{S(n,\chi_{\lambda^{(2)},q},\chi)}}{|G(R_{\ell,M})|} = S_1 + S_2,
\end{equation}
where
\begin{align}
S_1 &= \frac{\sum_{\chi \in G(R_{\ell,M}), \text{ primitive, odd and }\chi^2 \neq \chi_0} q^{-n} S(n,\chi_{\lambda^{(1)},q},\chi)\overline{S(n,\chi_{\lambda^{(2)},q},\chi)}}{|G(R_{\ell,M})|}, \nonumber \\
\label{sums2badchars} S_2 &=\frac{\sum_{\chi \in G(R_{\ell,M}) \setminus \{\chi_0\}, \text{ even or non-primitive or }\chi^2 = \chi_0} q^{-n} S(n,\chi_{\lambda^{(1)},q},\chi)\overline{S(n,\chi_{\lambda^{(2)},q},\chi)}}{|G(R_{\ell,M})|}.
\end{align}
From \S\ref{secchars}, Lemma~\ref{lemphi} and Lemma~\ref{lemquadchars}, we see that the number of characters appearing in the sum \eqref{sums2badchars} is $O_{\deg M}( \frac{|G(R_{\ell,M})|}{q})$. Thus, we may use the first part of Lemma~\ref{lemboundfact} to bound $S_2$ by
\begin{equation}\label{chiavs2bound}
S_2 = O_{n,\ell,\deg M}\Big(\frac{1}{q}\Big).
\end{equation}
We estimate $S_1$ as follows. From Lemma~\ref{lemschur}, we may rewrite $S_1$ as 
\begin{equation*}
S_1 = \frac{\sum_{\chi \in G(R_{\ell,M}), \text{ primitive, odd and }\chi^2 \neq \chi_0}  s_{\lambda^{(1)'}}(\Theta_{\chi})\overline{s_{\lambda^{(2)'}}(\Theta_{\chi})}}{|G(R_{\ell,M})|} + O_{n,\ell,\deg M}\Big(\frac{1}{\sqrt{q}}\Big).
\end{equation*}
Defining
\begin{equation*}
\begin{split}
S_3 &= \frac{\sum_{\chi \in G(R_{\ell,M}), \text{ primitive, odd}}  s_{\lambda^{(1)'}}(\Theta_{\chi})\overline{s_{\lambda^{(2)'}}(\Theta_{\chi})}}{|G(R_{\ell,M})|},\\
S_4 &= \frac{\sum_{\chi \in G(R_{\ell,M}), \text{ primitive, odd and }\chi^2 = \chi_0}  s_{\lambda^{(1)'}}(\Theta_{\chi})\overline{s_{\lambda^{(2)'}}(\Theta_{\chi})}}{|G(R_{\ell,M})|},
\end{split}
\end{equation*}
we have 
\begin{equation}\label{chiavs1s3s4}
S_1 = S_3 - S_4+ O_{n,\ell,\deg M}(\frac{1}{\sqrt{q}}).
\end{equation}
The sum $S_4$ is easily bounded by Lemma~\ref{lemphi} and Lemma~\ref{lemquadchars} as
\begin{equation}\label{chiavs4bound}
S_4 = O_{n,\ell,\deg M}(\frac{1}{q}).
\end{equation}
We now turn to $S_3$. The function $\Theta \mapsto s_{\lambda^{(1)'}}(\Theta)\overline{s_{\lambda^{(2)'}}(\Theta)}$ is a class function on $\mathrm{PU}(\ell+\deg M - 1)$, which is also a Laurent polynomial in the eigenvalues of $\Theta$, and so it can be written as a  linear combination of irreducible characters of $\mathrm{PU}(\ell+\deg M - 1)$:
\begin{equation}\label{chardecompss}
s_{\lambda^{(1)'}}(\Theta)\overline{s_{\lambda^{(2)'}}(\Theta)} = \sum_{\rho \text{ irred. rep. of }\mathrm{PU}(\ell+\deg M-1)} a_{\rho,\lambda^{(1)},\lambda^{(2)},\ell+\deg M -1} \cdot \Tr(\rho(\Theta)),
\end{equation}
where $a_{\rho,\lambda^{(1)},\lambda^{(2)},\ell+\deg M -1}= 0 $ for all but finitely many $\rho$-s. We denote the trivial representation of $\mathrm{PU}(\ell+\deg M - 1)$ by $\mathbf{1}$.
Since $\int_{\mathrm{PU}(\ell+\deg M-1)} \chi_{\rho}(\Theta) d\Theta = \indic_{\rho = \mathbf{1}}$, we can calculate the coefficient of $\mathbf{1}(\Theta)$ in the right hand side of \eqref{chardecompss} by integrating over $\mathrm{PU}(\ell+\deg M-1)$:
\begin{equation}\label{a1coef}
a_{\mathbf{1},\lambda^{(1)},\lambda^{(2)},\ell+\deg M -1} = \int_{\mathrm{PU}(\ell+\deg M-1)} 
s_{\lambda^{(1)'}}(\Theta)\overline{s_{\lambda^{(2)'}}(\Theta)} d\Theta = \indic_{\substack{\lambda^{(1)}=\lambda^{(2)} \\ \lambda_1^{(1)} \le \ell+\deg M - 1}},
\end{equation}
where the last passage is a special case of orthogonality of irreducible characters  \cite[Eq.~(3.8)]{gamburd2007}. Now, from \eqref{chardecompss}, \eqref{a1coef} and Theorem~\ref{thmequigenarith}, we obtain
\begin{equation}\label{chiavs3bound}
\begin{split}
\Big| S_3 - \indic_{\substack{\lambda^{(1)} = \lambda^{(2)}\\ \lambda^{(1)}_1 \le \ell+\deg M - 1}}  \Big| & \le \sum_{\mathbf{1} \neq \rho \text{ irred. rep. of }\mathrm{PU}(\ell+\deg M-1)} |a_{\rho,\lambda^{(1)},\lambda^{(2)},\ell+\deg M -1} |  \\
&\qquad \cdot \left| \frac{\sum_{\chi \in G(R_{\ell,M}) \text{ primitive and odd}} \Tr(\rho(\Theta_{\chi}))}{\# \{ \chi \in G(R_{\ell,M}) :  \chi \text{ primitive and odd }\}} \right| \cdot (1+O_{\ell,\deg M}(\frac{1}{q})) \\
& \le (1+O_{\ell,\deg M}(\frac{1}{q})) \sum_{\mathbf{1} \neq \rho \text{ irred. rep. of }\mathrm{PU}(\ell+\deg M-1)} |a_{\rho,\lambda^{(1)},\lambda^{(2)},\ell+\deg M -1} |  \frac{ C(\rho)}{\sqrt{q}} 
\end{split}
\end{equation}
for a constant $C(\rho)$ depending only on $\rho$, where in the first inequality we used \eqref{eq:number of chars}. Combining \eqref{avchis1s2def}, \eqref{chiavs2bound}, \eqref{chiavs1s3s4}, \eqref{chiavs4bound} and  \eqref{chiavs3bound}, we obtain \eqref{avchiind}, as needed.
\end{proof}
The following theorem is a generalization of \cite[Thm.~10.1]{rodgers2018}, which corresponds to the special case $\Delta=1$.
\begin{theorem}\label{thmcovgap}
Let $\alpha_q,\beta_q$ be factorization functions which come from $\alpha,\beta$. Let $n$ be a positive integer. Let $(\Delta,h) \in \mathcal{A}_q \times \mathbb{N}$ such that $\Delta$ is squarefree and either $n-4 \ge h \ge \deg \Delta$, or $\deg \Delta \ge 2$ and $h=n$. Then
\begin{equation}\label{covgen}
\mathrm{Cov}_{\mathcal{M}_q}(\alpha_q,\beta_q;n,\Delta,n-h) = q^{h-\deg \Delta} \sum_{\substack{\lambda \vdash n \\ \lambda_1 \le n-h+\deg \Delta -1}} \hat{\alpha}_{\lambda} \overline{\hat{\beta}_{\lambda}} + O_{\max(\alpha_q;n),\max(\beta_q;n),n}(q^{h-\deg\Delta - \frac{1}{2}}).
\end{equation}
\end{theorem}
\begin{proof}
From Proposition~\ref{avchiviaschur} with $\ell=n-h, \, M=\Delta$ and \eqref{covchi}, we have
\begin{equation*}
\begin{split}
\mathrm{Cov}_{\mathcal{M}_q}(\alpha_q,\beta_q;n,\Delta,n-h) &= \left(\frac{1}{q^{n-h} \phi(\Delta)} \right)^2 \sum_{\chi_0 \neq \chi \bmod R_{n-h,\Delta}} S(n,\alpha_q,\chi) \overline{S(n,\beta_q, \chi)}\\
&=\frac{q^h}{\phi(\Delta)} \frac{\sum_{\chi \in G(R_{n-h,\Delta})\setminus \{ \chi_0\}} q^{-n} S(n,\alpha_q,\chi)\overline{S(n,\beta_q,\chi)} }{|G(R_{n-h,\Delta})|}\\
&= \frac{q^h}{\phi(\Delta)} \Big(\sum_{\substack{\lambda \vdash n \\ \lambda_1 \le n-h+\deg \Delta -1}} \hat{\alpha}_{\lambda} \overline{\hat{\beta}_{\lambda}} + O_{\max(\alpha_q;n),\max(\beta_q;n),n,n-h,\deg \Delta}(\frac{1}{\sqrt{q}})\Big),
\end{split}
\end{equation*}
from which \eqref{covgen}  follows by invoking \eqref{weakphi}.
\end{proof}
Although we do not use it, the next proposition shows that in the case $h=\deg \Delta$, the error term in Theorem~\ref{thmcovgap} can be improved, and the proof is elementary.
\begin{proposition}\label{propcovgap}
Let $\alpha_q,\beta_q$ be factorization functions which come from $\alpha,\beta$. Let $n$ be a positive integer and let $\Delta$ be a polynomial in $\mathcal{A}_q$ of degree $\le n$. Then
\begin{equation*}
\mathrm{Cov}_{\mathcal{M}_q}(\alpha_q,\beta_q;n,\Delta,n-\deg \Delta) =  \sum_{\substack{\lambda \vdash n \\ \lambda_1 \le n-1}} \hat{\alpha}_{\lambda} \overline{\hat{\beta}_{\lambda}} + O_{\max(\alpha_q;n),\max(\beta_q;n),n}(\frac{1}{q}).
\end{equation*}
\end{proposition}
\begin{proof}
From \eqref{weakphi} we have 
\begin{equation}\label{condelfull2}
\begin{split}
\mathrm{Cov}_{\mathcal{M}_q}(\alpha_q,\beta_q;n,\Delta,n-\deg \Delta) &=\frac{\sum_{f \in \Mnq} \alpha_q(f) \overline{\beta_q(f)}}{q^{n}} - \frac{\sum_{f \in \Mnq} \alpha_q(f)}{q^{n}} \cdot \frac{\sum_{f \in \Mnq} \overline{\beta_q(f)}}{q^{n}}\\
&\qquad + O_{\max(\alpha_q;n),\max(\beta_q;n),n}\Big( \frac{1}{q}\Big).
\end{split}
\end{equation}
We use the notation of the proof of Lemma~\ref{lemcorbiggcd}, with $n$ in place of $n-1$. In particular, for a partition $\lambda$ of $n$, let $p_{\lambda}$ the probability that a uniformly chosen element in $S_{n}$ has cycle structure given by $\lambda$. Let $\gamma_q$ be a factorization function which comes from $\gamma$. By \eqref{equilamb} we have
\begin{equation}\label{sumviaplam}
\frac{\sum_{f \in \Mnq} \gamma_q(f)}{q^{n}} = \sum_{\lambda \vdash n} p_{\lambda} \cdot \gamma_q(\lambda) + O_{n,\max(\gamma_q;n)}(\frac{1}{q})  = \frac{1}{n!}\sum_{\pi \in S_n} \gamma(\pi) + O_{n,\max(\gamma_q;n)}(\frac{1}{q}).
\end{equation}
If $\gamma |_{S_n^{\#}} (\pi) = \sum_{\lambda \vdash n} \hat{\gamma}_{\lambda} \chi_{\lambda}(\pi)$ is the Fourier expansion of $\gamma$, then \eqref{sumviaplam} may be expressed as
\begin{equation}\label{sumviaharmonics}
\frac{\sum_{f \in \Mnq} \gamma_q(f)}{q^{n}} = \hat{\gamma}_{(n)}+ O_{n,\max(\gamma_q;n)}(\frac{1}{q}).
\end{equation}
From \eqref{condelfull2} and \eqref{sumviaharmonics} with $\gamma=\alpha,\overline{\beta}$, we have
\begin{equation*}
\mathrm{Cov}_{\mathcal{M}_q}(\alpha_q,\beta_q;n,\Delta,n-\deg \Delta) = \EE_{S_n} \alpha \overline{\beta}  - \hat{\alpha }_{(n)} \cdot \overline{\hat{\beta}_{(n)}} +O_{\max(\alpha_q;n),\max(\beta_q;n),n}\Big( \frac{1}{q}\Big).
\end{equation*}
The proof is concluded by noting that
\begin{equation*}
\EE_{S_n} \alpha \overline{\beta}  = \sum_{\lambda \vdash n} \hat{\alpha}_{\lambda} \overline{\hat{\beta}_{\lambda}},
\end{equation*}
a direct consequence of Plancherel theorem for the group $S_n$.
\end{proof}
We also need the following identity.
\begin{lemma}\label{lem:CoeffInterpret}
Let $\alpha_q$ be a factorization function which comes from $\alpha$. Let $n$ be a positive integer. Then
\begin{equation*}
 \hat{\alpha}_{(n-1,1)} = 
 \EE_{f \in \mathcal{A}_{n-1,q}} \alpha_q(f \cdot T) - \EE_{\Anq} \alpha_q+O_{n,\max(\alpha_q;n)}(\frac{1}{q}).
\end{equation*}
\end{lemma}
\begin{proof}
Let $x \in \FF_q$. Observe that, whenever $ f\in \mathcal{A}_{n-1,q}$ is squarefree and coprime to $T$ and $T-x$, $\alpha_q(f \cdot T)= \alpha_q (f \cdot (T-x))$. Hence
\[  \EE_{f \in \mathcal{A}_{n-1,q}}\alpha_q(f \cdot T) =  \EE_{f \in \mathcal{A}_{n-1,q},\, x \in \FF_q}\alpha_q(f \cdot (T-x))  + O_{n,\max(\alpha_q;n)}(\frac{1}{q} ). \] 
Now
\begin{equation*}
\begin{split}
\EE_{f \in \mathcal{A}_{n-1,q},\, x \in \FF_q}\alpha_q(f \cdot (T-x))  &= \frac{1}{q^n(q-1)} \sum_{ f \in \mathcal A_{n-1,q}, x\in \mathbb F_q} \alpha_q( f \cdot (T-x) ) \\
&=\frac{1}{q^n(q-1)} \sum_{g \in \mathcal A_{n, q} } \sum_{ \substack{ x\in \mathbb F_q:\\ T-x \mid g}} \alpha_q(g)  =\frac{1}{q^n(q-1)}  \sum_{ g \in \mathcal A_{n,q}} \alpha_q(g) \# \{ x \in \mathbb F_q  : g(x)=0 \} .
\end{split}
\end{equation*}
Hence
\[ \EE_{f \in \mathcal{A}_{n-1,q}} \alpha_q(f \cdot T) - \EE_{\Anq} \alpha_q = \frac{1}{q^n(q-1)} \sum_{ g \in \mathcal A_{n,q}} \alpha_q(g) \left(  \# \{ x \in \mathbb F_q  : g(x)=0 \}  - 1 \right)  + O_{n,\max(\alpha_q;n)}(\frac{1}{q} ). \] 

Now for $g$ a squarefree polynomial, the function $\left( \# \{ x \in \mathbb F_q  : g(x)=0  \} - 1\right)$ simply counts the fixed points of the corresponding permutation and subtracts one. This is also the trace of the standard $(n-1)$-dimensional representation of $S_n$. By \cite[Lemma~2.1]{andrade2015}, the average value of $\alpha_q$ times the trace of the standard representation over squarefree polynomials is equal, to within $O_{n,\max(\alpha_q;n)}(1/q)$, of its average over permutations, which is equal, by character theory, to the multiplicity $\hat{\alpha}_{(n-1,1)}$ of the standard representation within $\alpha$, as needed.

\end{proof}

\subsection{Conclusion of the proof of Theorem~\ref{maintermcovthm}}
Recall that $n\ge 5$ is an integer and that $\Delta$ is a squarefree polynomial in $\mathcal{A}_q$ which is either of degree~$\le n-5$ or of degree $n-1$. From the second part of Proposition~\ref{twoidengap} with $h=\deg \Delta+1$, we obtain
\begin{equation}\label{propwithhchosen}
\begin{split}
&\frac{\mathrm{Cov}_{\mathcal{M}_q}(\alpha,\beta;n,\Delta,n-\deg \Delta - 1)}{q} - \mathrm{Cov}_{\mathcal{M}_q}(\alpha,\beta;n,\Delta,n-\deg \Delta)\\ 
& \quad =\sum_{c \in \FF_q^{\times}} \mathrm{Cov}_{f \in \mathcal{M}_{n,q,\Delta}} (\alpha(f), \beta(f+c \Delta)).
\end{split}
\end{equation}
From \eqref{propwithhchosen} and Proposition~\ref{propcordel}, we obtain
\begin{equation}\label{twocovsvscov}
\begin{split}
&\frac{\mathrm{Cov}_{\mathcal{M}_q}(\alpha,\beta;n,\Delta,n-\deg \Delta - 1)}{q} - \mathrm{Cov}_{\mathcal{M}_q}(\alpha,\beta;n,\Delta,n-\deg \Delta)\\
&= q \cdot \mathrm{Cov}_{\mathcal{A}_q}(\alpha,\beta;n,\Delta)  -a_{\Delta,q} \Big( \EE_{\Anq} \alpha - \EE_{f\in \mathcal{A}_{n-1,q}} \alpha(f \cdot T) \Big) \cdot \Big( \EE_{\Anq} \overline{\beta} - \EE_{f\in \mathcal{A}_{n-1,q}} \overline{\beta}(f \cdot T) \Big)\\
&\qquad +  O_{\mathrm{max} (\alpha;n),\mathrm{max} (\beta;n),n}\Big( \frac{1}{\sqrt{q}} \Big).
 \end{split} 
\end{equation}
From \eqref{twocovsvscov}, Lemma~\ref{lem:CoeffInterpret} and Theorem~\ref{thmcovgap} with $h=1+\deg \Delta$ and $h=\deg \Delta$, we obtain 
 \begin{equation}\label{covnonisolated2}
 \begin{split}
 & -\hat{\alpha}_{(n-1,1)} \overline{\hat{\beta}_{(n-1,1)}} + O_{\max(\alpha_q;n),\max(\beta_q;n),n}\Big(\frac{1}{\sqrt{q}}\Big)\\
 &\quad = q \cdot \mathrm{Cov}_{\mathcal{A}_q}(\alpha,\beta;n,\Delta)  -a_{\Delta,q} \hat{\alpha}_{(n-1,1)} \overline{\hat{\beta}_{(n-1,1)}} +  O_{\mathrm{max} (\alpha_q;n),\mathrm{max} (\beta_q;n),n}\Big( \frac{1}{\sqrt{q}} \Big).
  \end{split}.
 \end{equation}
By isolating the term $\mathrm{Cov}_{\mathcal{A}_q}(\alpha,\beta;n,\Delta)$ in \eqref{covnonisolated2}, we conclude the proof of the theorem. \qed
\section{Proof of Theorem~\ref{thmmn}}
\subsection{Fundamental identity}
The following proposition uses the definitions of $S(n,\alpha,\chi)$ and $\mathcal{M}_{n,q,\Delta}$, recall \eqref{defsn} and \eqref{defmnd}.
\begin{proposition}\label{lemcorexp2}
Let $\alpha, \beta \colon \mathcal{M}_q \to \mathbb{C}$ be two arithmetic functions. Let $n$ be a positive integer and let $\Delta$ be a polynomial in $\mathcal{A}_q$ of degree $<n$. Let $c \in \FF_q^{\times}$ be the leading coefficient of $\Delta$, and $g$ be the unique element of $(\mathcal{M}_q/R_{n-\deg \Delta,\Delta})^{\times}$ such that
\begin{equation*}
g \equiv 1 \bmod \Delta, \quad g \equiv T^{n-\deg \Delta}+c \bmod R_{n-\deg \Delta, 1}.
\end{equation*}
We have
\begin{equation*}
\EE_{f \in \mathcal{M}_{n,q,\Delta}} \alpha(f) \overline{\beta(f+\Delta)} = \frac{1}{q^{2(n-\deg \Delta)} \phi^2(\Delta)} \sum_{\chi \in G(R_{n-\deg \Delta,\Delta}) } \chi(g) S(n,\alpha,\chi) \overline{S(n,\beta,\chi)}.
\end{equation*}
\end{proposition}
\begin{proof}
The orthogonality relation \eqref{ortho2} implies that for any $f \in \mathcal{M}_{n,q,\Delta}$ we have 
\begin{equation}\label{findhay}
\alpha(f) = \frac{1}{q^{n-\deg \Delta}  \phi(\Delta)} \sum_{\substack{ g_1 \in \mathcal{M}_{n,q,\Delta}  \\ \chi \in G(R_{n-\deg \Delta, \Delta})}} \alpha(g_1) \chi(g_1) \overline{\chi(f)},
\end{equation}
and similarly we have
\begin{equation}\label{fcindhay}
\beta(f+\Delta) = \frac{1}{q^{n- \deg \Delta} \phi(\Delta)} \sum_{\substack{ g_2 \in \mathcal{M}_{n,q,\Delta} \\ \chi \in G(R_{n-\deg \Delta})}} \beta(g_2) \chi(g_2) \overline{\chi(f+\Delta )}.
\end{equation}
From \eqref{findhay} and \eqref{fcindhay} we obtain
\begin{equation}\label{simplicor1hay}
\begin{split}
& \EE_{f \in \mathcal{M}_{n,q,\Delta}} \alpha(f) \overline{\beta(f+\Delta)} \\
&\qquad = \frac{1}{q^{3(n-\deg \Delta)} \phi^3(\Delta)} \sum_{\substack{g_1,g_2 \in \mathcal{M}_{n,q,\Delta}  \\ \chi_1, \chi_2 \in G(R_{n-\deg \Delta, \Delta})}} \Big( \alpha(g_1) \chi_1(g_1) \overline{\beta(g_2)\chi_2(g_2)} \sum_{ f \in \mathcal{M}_{n,q,\Delta}} \overline{\chi_1(f)}\chi_2(f+\Delta) \Big).
\end{split}
\end{equation}
By the definition of $g$, we have 
\begin{equation}\label{fdelgrel}
f+\Delta \equiv g \cdot f \bmod R_{n-\deg \Delta,\Delta}.
\end{equation}
for all $f\in \mathcal{M}_{n,q,\Delta}$. For any pair of characters $\chi_1,\chi_2 \in G(R_{n-\deg \Delta,\Delta})$ we have, by \eqref{fdelgrel} and the orthogonality relation \eqref{ortho1} with $F=\mathcal{M}_{n,q,\Delta}$,
\begin{equation}\label{corchihay}
\begin{split}
\sum_{f \in \mathcal{M}_{n,q,\Delta}} \overline{\chi_1(f)}\chi_2(f+\Delta) &= \chi_2(g) \sum_{f \in \mathcal{M}_{n,q,\Delta}} \overline{\chi_1(f)}\chi_2(f) \\
&= \chi_2(g) \cdot q^{n-\deg \Delta}  \phi(\Delta) \cdot \indic_{\chi_1 = \chi_2}.
\end{split}
\end{equation}
Plugging \eqref{corchihay} in \eqref{simplicor1hay}, we conclude the proof.
\end{proof}
\subsection{Estimates}
\begin{lemma}\label{lemmntomnq}
Let $\alpha \colon \mathcal{M}_q \to \mathbb{C}$. Let $n$ be a positive integer and let $\Delta$ be a polynomial in $\mathcal{A}_q$ of degree $<n$. We have
\begin{equation*}
\EE_{\mathcal{M}_{n,q,\Delta}} \alpha =
\EE_{\Mnq} \alpha  + O_{n,\max(\alpha;n)}\Big(\frac{1}{q}\Big).
\end{equation*}
\end{lemma}
\begin{proof}
We have
\begin{equation}\label{mnqmnqdel0}
\EE_{\mathcal{M}_{n,q,\Delta}} \alpha -
\EE_{\Mnq} \alpha   = S_1+S_2,
\end{equation}
where
\begin{equation*}
S_1 = (\frac{1}{q^{n-\deg \Delta} \phi(\Delta)} - \frac{1}{q^n}) \sum_{f \in \mathcal{M}_{n,q,\Delta}} \alpha(f), \qquad S_2 = -\frac{1}{q^n} \sum_{f \in \Mnq \setminus \mathcal{M}_{n,q,\Delta}} \alpha(f).
\end{equation*}
We have, by \eqref{weakphi},
\begin{equation}\label{mnqmnqdel1}
|S_1|, |S_2| \le (1-\frac{\phi(\Delta)}{|\Delta|}) \max(\alpha;n) = O_{n,\max(\alpha;n)}\Big(\frac{1}{q}\Big).
\end{equation}
We conclude the proof from \eqref{mnqmnqdel0} and \eqref{mnqmnqdel1}.
\end{proof}

\begin{lemma}\label{reducprimodd}
Let $\alpha, \beta \colon \mathcal{M}_q \to \mathbb{C}$. Let $n$ be a positive integer and let $\Delta$ be a polynomial in $\mathcal{A}_q$ of degree $<n$. Let $g \in (\mathcal{M}_q/R_{n-\deg \Delta,\Delta})^{\times}$.
We have
\begin{equation}\label{simpsumchig}
\begin{split}
\frac{1}{q^{2(n-\deg \Delta)} \phi^2(\Delta)}& \sum_{\chi \in G(R_{n-\deg \Delta,\Delta}) } \chi(g) S(n,\alpha,\chi) \overline{S(n,\beta,\chi)} \\
&=  \EE_{\Mnq} \alpha \cdot \EE_{\Mnq} \overline{\beta} \\
& \qquad +  \frac{1}{q^{2(n-\deg \Delta)} \phi^2(\Delta)} \sum_{\substack{\chi \in G(R_{n-\deg \Delta,\Delta})\\\text{odd and primitive} }} \chi(g) S(n,\alpha,\chi) \overline{S(n,\beta,\chi)}\\
& \qquad +  O_{n,\max(\alpha;n),\max(\beta;n)}\Big(\frac{1}{q}\Big).
\end{split}
\end{equation}
\end{lemma}
\begin{proof}
The term corresponding to $\chi=\chi_0$ in the left hand side of \eqref{simpsumchig} contributes, by Lemma~\ref{lemmntomnq},
\begin{equation*}
\EE_{\mathcal{M}_{n,q,\Delta}} \alpha \cdot \EE_{\mathcal{M}_{n,q,\Delta}} \overline{\beta} = \EE_{\Mnq} \alpha \cdot \EE_{\Mnq} \overline{\beta} + O_{n,\max(\alpha;n),\max(\beta;n)}\Big(\frac{1}{q}\Big).
\end{equation*}
The terms corresponding to $\chi \in G(R_{n-\deg \Delta,\Delta}) \setminus \{\chi_0\}$ in the left hand side of \eqref{simpsumchig}, which are even or not primitive, contribute, by \S\ref{secchars} and the first part of Lemma~\ref{lemboundfact},
\begin{equation*}
O_{n,\max(\alpha;n),\max(\beta;n)}\Big(\frac{1}{q}\Big).
\end{equation*}
These two estimates conclude the proof of the lemma.
\end{proof}

\subsection{Hidden symmetry}
The following key proposition introduces an action of $\FF_q^{\times}$ on $G(R_{\ell,1})$, which preserves primitivity and $L$-functions.
\begin{proposition}\label{propsym}
Let $\ell$ be a positive integer. Let $\chi \in G(R_{\ell,1})$ be a primitive character. For any $c \in \FF_q^{\times}$, define a function $\chi_{c} \colon \mathcal{M}_q \to \mathbb{C}$ by
\begin{equation*}
\chi_{c}(f) = \chi(f(c T)/{c}^{\deg f}).
\end{equation*}
Then $\chi_{c}$ is well defined on $\mathcal{M}_q / R_{\ell,1}$ and in fact is a primitive character in $G(R_{\ell, 1})$. Moreover,
\begin{equation*}
\Theta_{\chi} = \Theta_{\chi_{c}}.
\end{equation*}
\end{proposition}

\begin{proof}
Fix $c \in \FF_q^{\times}$. Let $f_1,f_2 \in \mathcal{M}_q$ be polynomials such that $f_1 \equiv f_2 \bmod R_{\ell, 1}$. Then $f_1,f_2$ have the same first $\ell$ next-to-leading coefficients. The $i$-th next to leading coefficient of $f_j(c T)/{c}^{\deg f_j}$ ($j \in \{1,2\}$) is the $i$-th next-to-leading coefficient of $f_j(T)$, divided by ${c}^i$. Thus, $f_1(c T)/{c}^{\deg f_1} \equiv  f_2(c T)/{c}^{\deg f_2} \bmod R_{\ell,1}$. This shows that $\chi_{c}$ can be regarded as a function of $\mathcal{M}_q / R_{\ell,1}$. By definition, $\chi_{c}$ is multiplicative, and it takes $1$ to $1$, so $\chi_{c} \in G(R_{\ell,1})$. 

We now establish $\Theta_{\chi} = \Theta_{\chi_{c}}$. The coefficients of $u^i$ in $L(u,\chi)$  and $L(u,\chi_{c})$ are given by $\sum_{f \in \mathcal{M}_{i,q}} \chi(f(T))$ and $\sum_{f \in \mathcal{M}_{i,q}} \chi(f(c T)/{c}^i)$, respectively. The map $f \mapsto f(c T)/{c}^i$ is a permutation of $\mathcal{M}_{i,q}$, whose inverse is given by $f \mapsto f(T/c){c}^i$. Thus, $L(u,\chi) = L(u,\chi_{c})$ and the corresponding matrices must coincide. As $\deg L(u,\chi_{c}) = \deg L(u,\chi) = \ell-1$, it follows that $\chi_{c}$ is a primitive character.
\end{proof}
\begin{lemma}\label{sumsunderlambda}
Let $\ell$ be a positive integer. Let $\chi \in G(R_{\ell,1})$. Let $c \in \FF_q^{\times}$. For any factorization function $\alpha$, we have
\begin{equation*}
S(n,\alpha,\chi) = S(n,\alpha,\chi_{c}).
\end{equation*}
\end{lemma}
\begin{proof}
Since $f(T), f(c T)/c^{\deg f}$ have the same extended factorization type for any $c \in \FF_q^{\times}$, and the inverse of $f \mapsto f(c T)/c^{\deg f}$ is $f \mapsto f(T/c)c^{\deg f}$, we have
\begin{equation*}
\begin{split}
S(n,\alpha,\chi_{c})&=\sum_{f \in \Mnq} \alpha(f) \chi( f(c T)/c^n) 
=\sum_{f \in \Mnq} \alpha(f(T/c)c^n) \chi(f)\\ 
&= \sum_{f \in \Mnq} \alpha(f) \chi( f) = S(n,\alpha,\chi),
\end{split}
\end{equation*}
as needed.
\end{proof}
\subsection{Conclusion of proof}
Applying Proposition~\ref{lemcorexp2} with $\Delta \in \FF_q^{\times}$, we find that 
\begin{equation}\label{propwithdel1}
\EE_{f\in \Mnq} \alpha_q(f) \overline{\beta_q(f+\Delta)} = \frac{1}{q^{2n}} \sum_{\chi \in G(R_{n,1}) } \chi(g) S(n,\alpha_q,\chi) \overline{S(n,\beta_q,\chi)},
\end{equation}
where $g$ may be taken to be $g=T^n +\Delta$. Applying Lemma~\ref{reducprimodd} to the right hand side of \eqref{propwithdel1}, we find that
\begin{equation}\label{propwithdel2}
\mathrm{Cov}_{\mathcal{M}_q}(\alpha_q,\beta_q;n,\Delta) =  \frac{1}{q^{2n}} \sum_{\substack{\chi \in G(R_{n,1})\\\text{primitive} }} \chi(T^n+\Delta) S(n,\alpha_q,\chi) \overline{S(n,\beta_q,\chi)}+ O_{n,\max(\alpha_q;n),\max(\beta_q;n)}\Big(\frac{1}{q}\Big) .
\end{equation}
We claim that the multiset $A=\{ \chi_{c} : c \in \FF_q^{\times}, \chi \in G(R_{n,1})\}$ consists of $q-1$ copies of $G(R_{n,1})$. Indeed, the map $\chi \mapsto \chi_{c}$ is a bijection for any $c \in \FF_q^{\times}$. Thus, in \eqref{propwithdel2} we may sum over primitive characters in $A$ and divide by $q-1$, instead of summing over primitive characters in $G(R_{n,1})$, and obtain from Lemma~\ref{sumsunderlambda}: 
\begin{equation}\label{covtwistedav}
\mathrm{Cov}_{\mathcal{M}_q}(\alpha_q,\beta_q;n,\Delta)= \frac{\sum_{\substack{\chi \in G(R_{n,1})\\\text{primitive} }} \frac{\sum_{c \in \FF_q^{\times}}\chi_{c}(T^n+\Delta)}{q-1} S(n,\alpha_q,\chi) \overline{S(n,\beta_q,\chi)}}{q^{2n}} +O_{n,\mathrm{max} (\alpha ; n),\mathrm{max} (\beta ; n)}\Big(\frac{1}{q} \Big).
\end{equation}
When $\chi \in G(R_{n,1})$, $\psi_{\chi}(x):=\chi(T^n+x)$ is an additive character of $\FF_q$, since $(T^n+x_1)(T^n+x_2)\equiv T^n+x_1+x_2 \bmod R_{n,1}$ and $(T^n+x)^p \equiv 1 \bmod R_{n,1}$. Moreover, we claim that if $\chi$ is primitive then $\psi_{\chi}$ is non-trivial. Otherwise, whenever $f \equiv g \bmod R_{n-1,1}$ we may write $f \equiv g \cdot (T^n+x) \bmod R_{n,1}$ for some $x \in \FF_q$, and then $\chi(f)=\chi(g)\chi(T^n+x)=\chi(g)$, implying $\chi$ is not primitive, a contradiction. Thus, if we set
\begin{equation*}
A(\chi,\Delta)=\frac{\sum_{c \in \FF_q^{\times}}\chi_{c}(T^n+\Delta)}{\sqrt{q}}=\frac{\sum_{c \in \FF_q^{\times}}\chi(T^n+\frac{\Delta}{c^n})}{\sqrt{q}}=\frac{\sum_{c \in \FF_q^{\times}}\psi_{\chi}(\frac{\Delta}{c^n})}{\sqrt{q}}=\frac{\sum_{c \in \FF_q^{\times}}\psi_{\chi}(\Delta c^n)}{\sqrt{q}},
\end{equation*}
then $|A(\chi,\Delta)| \le n$ by Weil's bound on additive character sums \cite[Thm.~2E]{schmidt1976}. We express \eqref{covtwistedav} as
\begin{equation}\label{covtwistedav2}
\mathrm{Cov}_{\mathcal{M}_q}(\alpha_q,\beta_q;n,\Delta)=\frac{\sqrt{q}}{q-1} \frac{\sum_{\substack{\chi \in G(R_{n,1})\\\text{primitive} }} A(\chi,\Delta) S(n,\alpha_q,\chi) \overline{S(n,\beta_q,\chi)}}{q^{2n}} +O_{n,\mathrm{max} (\alpha ; n),\mathrm{max} (\beta ; n)}\Big(\frac{1}{q} \Big).
\end{equation}
The number of characters over which we sum and satisfy $\chi^2 =\chi_0$ is $O(q^{\frac{n-1}{2}})$ be Lemma~\ref{lemquadchars}, and so by Lemma~\ref{lemboundfact}, their total contribution to the right hand side of \eqref{covtwistedav2} is $O_{\max(\alpha_q;n),\max(\beta_q;n),n}(q^{-\frac{n}{2}})$, which can be absorbed in the error term. From now on we ignore these characters when it will be convenient for us.

If either $\alpha_q$ or $\beta_q$ is supported on non-squarefrees, then $S(n,\alpha_q,\chi) \overline{S(n,\beta_q,\chi)} = O_{\max(\alpha_q;n),\max(\beta_q;n),n}(q^{n-\frac{1}{2}})$ by Lemma~\ref{lemboundfact}, and the right hand side of \eqref{covtwistedav2} is $O_{\max(\alpha_q;n),\max(\beta_q;n),n}(\frac{1}{q})$, as needed. Thus, since both $\EE_{f\in \Mnq} \alpha_q(f) \overline{\beta_q(f+\Delta)}$ and $S(n,\alpha_q,\chi) \overline{S(n,\beta_q,\chi)}$ are linear in $\alpha_q$ and conjugate-linear in $\beta_q$, it suffices to consider the case that $\alpha_q = \chi_{\lambda_1,q}$ and $\beta_q = \chi_{\lambda_2,q}$, where $\lambda_1,\lambda_2 \vdash n$. We then have, by Lemma~\ref{lemschur}, $
S(n,\alpha_q,\chi) = q^{n/2} s_{\lambda_1'}(\Theta_{\chi}) +O_{n}(q^{\frac{n-1}{2}})$ and $S(n,\beta_q,\chi) = q^{n/2} s_{\lambda_2'}(\Theta_{\chi}) + O_n(q^{\frac{n-1}{2}})$, and it remains to show that
\begin{equation}\label{eq:rhoscancel}
 q^{-n}\sum_{\substack{\chi \in G(R_{n,1})\\\text{primitive} }} A(\chi,\Delta)(\chi) s_{\lambda_1^{'}}(\Theta_{\chi})\overline{s_{\lambda_2^{'}}(\Theta_{\chi})} =O_{n,\mathrm{max} (\alpha ; n),\mathrm{max} (\beta ; n)}\Big(\frac{1}{q} \Big).
\end{equation}
By decomposing $s_{\lambda_1'} \overline{s_{\lambda_2'}}$ as a linear combination of irreducible characters of $\mathrm{PU}(n-1)$, we may apply Theorem~\ref{thmequishortgauss} with $\ell=n$ and conclude that \eqref{eq:rhoscancel} holds, which concludes the proof of the theorem. \qed
\section{Applications}
\subsection{Proof of Theorem~\ref{thmp}}\label{secprofthmp}
From Theorems~\ref{maintermcovthm} and \ref{thmmn} applied to the functions $\alpha_q=\beta_q=\Lambda_q$, together with the calculation of the constants in Corollary \ref{corhighercoeff}, we obtain the results of Theorem~\ref{thmp}. \qed
\subsection{Proof of Theorem~\ref{thmmu}}\label{secprofthmmu}
From Theorems~\ref{maintermcovthm} and \ref{thmmn} applied to the functions $\alpha_q=\beta_q=\mu_q$, together with the calculation of the constants in Corollary \ref{corhighercoeff}, we obtain the results of Theorem~\ref{thmmu}. \qed
\subsection{Proof of Theorem~\ref{thmd}}\label{secprofthmd}
From Theorems~\ref{maintermcovthm} and \ref{thmmn} applied to the functions $\alpha_q=d_{k,q}$ and $\beta_q=d_{l,q}$, together with the calculation of the constants in Corollary \ref{corhighercoeff}, we obtain the results of Theorem~\ref{thmd}. \qed 
\subsection{Consistency of Theorem~\ref{thmp} with the Hardy-Littlewood Conjecture}\label{twprdisc}
Let $\Lambda$ be the usual von Mangoldt function, defined on the positive integers. The Hardy-Littlewood Conjecture \cite{hardy1923} states that for any even, non-zero integer $\Delta$, 
\begin{equation}\label{hlconj}
\frac{\sum_{n \le x}\Lambda(n)\Lambda(n+\Delta)}{x} \sim \mathfrak{S}_{\Delta}, \qquad (x \to \infty),
\end{equation}
where the constant $\mathfrak{S}_{\Delta}$ is defined as the following product over primes, which converges to a positive number:
\begin{equation*}
\mathfrak{S}_{\Delta} = \prod_{p \mid \Delta} \frac{1-\frac{1}{p}}{(1-\frac{1}{p})^2} \prod_{p \nmid \Delta} \frac{1-\frac{2}{p}}{(1-\frac{1}{p})^2}.
\end{equation*}
In the function field setting, for a prime power $q>2$, the same heuristics that suggest \eqref{hlconj} also suggest that
\begin{equation}\label{hlconjff}
\frac{\sum_{f \in \Mnq}\Lambda_q(f)\Lambda_q(f+\Delta)}{q^n} \sim  \mathfrak{S}_{\Delta,q}, \qquad (q^n \to \infty),
\end{equation}
for all $\Delta \in \mathcal{A}_q$, where $\mathfrak{S}_{\Delta,q}$ is a product over prime polynomials:
\begin{equation*}
\mathfrak{S}_{\Delta,q} = \prod_{P \mid \Delta} \frac{1-\frac{1}{|P|}}{(1-\frac{1}{|P|})^2} \prod_{P \nmid \Delta} \frac{1-\frac{2}{|P|}}{(1-\frac{1}{|P|})^2}.
\end{equation*}
Since $\mathfrak{S}_{\Delta,q}$ does not change if we multiply $\Delta$ by a non-zero scalar, \eqref{hlconjff} implies that
\begin{equation}\label{squareroothlffnm}
\frac{\sum_{f \in \Anq}\Lambda_q(f)\Lambda_q(f+\Delta)}{q^n(q-1)} \sim \mathfrak{S}_{\Delta,q}, \qquad (q^n \to \infty)
\end{equation}
for all $\Delta \in \mathcal{A}_q$. We have the following estimate for $\mathfrak{S}_{1,q}$:
\begin{equation*}
\begin{split}
\mathfrak{S}_{1,q} &= \prod_{P \in \mathcal{P}_q} \frac{1-\frac{2}{|P|}}{(1-\frac{1}{|P|})^2} = \prod_{P} (1-\frac{1}{|P|^2} + O(\frac{1}{|P|^3})) \\
&= \prod_{ i \ge 1} \prod_{P: \deg P=i}(1-\frac{1}{q^{2i}} + O(\frac{1}{q^{3i}})) = (1-\frac{1}{q^2}+O(\frac{1}{q^3}))^q \prod_{i \ge 2} (1+O(q^{-2i}))^{O(q^i)}\\
&= (1-\frac{1}{q} + O(\frac{1}{q^2})) (1+O(\frac{1}{q^2})) = 1-\frac{1}{q} + O(\frac{1}{q^2}),
\end{split}
\end{equation*}
from which we deduce that
\begin{equation}\label{cdelest}
\begin{split}
\mathfrak{S}_{\Delta,q} &= \prod_{P\mid \Delta} \frac{1-\frac{1}{|P|}}{1-\frac{2}{|P|}} \cdot \mathfrak{S}_{1,q} = \prod_{P \mid \Delta, \deg P >1} (1+O(\frac{1}{q^2})) \prod_{P \mid \Delta, \deg P=1} \frac{1-\frac{1}{q}}{1-\frac{2}{q}} \cdot (1-\frac{1}{q} + O(\frac{1}{q^2})) \\
&= (1+O_{\deg \Delta}(q^{-2})) (1+\frac{a_{\Delta,q}}{q} + O_{\deg \Delta}(q^{-2}))(1-\frac{1}{q} +O(\frac{1}{q^2})) \\
&= 1+\frac{-1+a_{\Delta,q}}{q} + O_{\deg \Delta}(q^{-2}).
\end{split}
\end{equation}
The first two terms of the Taylor series of $\mathfrak{S}_{\Delta,q}$, given in \eqref{cdelest}, agree with the main term given in Theorem~\ref{thmp} for $\Lambda_q$.

\subsection{Proof of Theorem~\ref{corolvar}}\label{secthmsumcoeff}
Plugging $\Delta=1$ and $h$ in place of $h+1$ in the second part of Proposition~\ref{twoidengap}, we obtain 
\begin{equation}\label{covdiff2}
\sum_{\delta \in \mathcal{A}_{h,q}} \mathrm{Cov}_{\mathcal{M}_q}(\alpha_q,\beta_q;n,\delta) = \frac{\mathrm{Cov}_{\mathcal{M}_q}(\alpha_q,\beta;n,1,n-h-1)}{q^{h+1}} - \frac{\mathrm{Cov}_{\mathcal{M}_q}(\alpha_q,\beta;n,1,n-h)}{q^{h}}
\end{equation}
for any $0 \le h \le n-1$. The proof is concluded by using Theorem~\ref{thmcovgap} to simplify the right hand side of \eqref{covdiff2}. \qed

\appendix
\section{Equidistribution results}\label{app:secequi}

\begin{theorem}\label{thmequigenarith}
Let $(\ell,M) \in \mathbb{N} \times \mathcal{A}_q$ such that $M$ is squarefree and either $\ell \ge 4$, or $\ell=0$ and $\deg M \ge 2$. Let $\rho$ be an irreducible non-trivial representation of $\mathrm{PU}(\ell+\deg M - 1)$. Then there exists a positive constant $C(\rho)$, depending only on $\rho$, such that
\begin{equation*}
\left| \frac{\sum_{\chi \in G(R_{\ell,M}) \text{ primitive and odd}} \Tr(\rho(\Theta_{\chi}))}{\# \{ \chi \in G(R_{\ell,M}) :  \chi \text{ primitive and odd} \}} \right| \le \frac{C(\rho)}{\sqrt{q}}.
\end{equation*}
\end{theorem}
The cases $\ell=0$ and $M=1$ are due to Katz \cite[Cor.~6.9]{katz2013}, \cite[\S8]{WVQKR}.

\begin{proof} Because the case $\ell=0$ is due to Katz, it suffices to handle the case $\ell \ge 4$.

We apply the isomorphism \eqref{structunit}. We can write the average over $\chi$ as an iterated average over, first, characters $\chi_1$ of $(\FF_q[T] / M \FF_q[T] )^{\times}$ of, second, an average over characters $\Lambda $ of $(1+T\FF_q[T])/(1+T^{\ell+1}\FF_q[T])$ and it suffices to prove the same bound for the average over $\Lambda$. We can view $\chi_1$ as a character of the idele class group of $\FF_q(T)$ unramified away from $M \infty$ (with at most tame ramification at $\infty$), hence a character of the Galois group of $\FF_q(T)$ unramified away from $M \infty$, which we view as a rank one Galois representation $V_{\chi_1} $. Let $\mathcal F_{\chi_1}$ be the rank one middle extension sheaf on $\mathbb A^1_{\FF_q}$ associated to $V_{\chi_1}$. 

In \cite[Theorem 1.2 and Theorem 1.3]{witttwists}, families of conjugacy classes $\varphi_{\Lambda}$ in the unitary group $\mathrm{U}(N) = \mathrm{U}(\ell + \deg M -1)$  associated, respectively, to $V_{\chi_1}$ and $\mathcal F_{\chi_1}$ are defined. It is proven in \cite[\S3, proof of Theorem 1.2]{witttwists} that these are equal, so we will use them interchangeably.  Moreover, these conjuacy classes match the ones $\Theta_\chi$ we have defined, i.e. \[\varphi_{\Lambda} = \Theta_\chi\] up to conjugacy. This is because conjugacy classes in $\mathrm{U}(\ell+ \deg M - 1)$ are uniquely determined by their characteristic polynomial, both conjugacy classes are defined such that their characteristic polynomials match certain $L$-functions, and these $L$-functions agree because the $L$-function of a Galois representation associated to a character equals the $L$-function of the character. Hence we can apply the equidistribution result of \cite[Theorem 1.3]{witttwists}. More precisely, we will apply its proof. 

Because $\rho$ is a non-trivial representation of a projective unitary group, it is not one-dimensional. In  \cite[\S6, proof of Theorem 1.3]{witttwists}, it is shown that

\[\left| \frac{\sum_{\Lambda \in S_q} \Tr(\rho(\varphi_{\Lambda}))}{\# \{ \chi \in G(R_{\ell,M}) :  \chi \text{ primitive and odd} \}} \right| =  O\left(   \frac{1}{\sqrt{q}}\right)\]

where $S_q$ is the set of primitive characters of $$(1+T\FF_q[T])/(1+T^{\ell+1}\FF_q[T]),$$ and, as mentioned earlier, $\varphi_{\Lambda}=\theta_\chi$.

However, it is not proved in \cite[\S6]{witttwists} that the constant in the big $O$ is uniform in the characteristic or the choice of $\chi_1$. We do this now using an analogue of the argument in \cite[Lemmas 2.7, 2.8, and 2.9]{representation-theory-moments}.

The constant arises as a sum of Betti numbers of \[ H^i_c(U_{\FF_q },\mathcal V)\] where $\mathcal V$ is a sheaf constructed from the representation $\rho$ and $U_{\mathbb F_q}$ is the open subset of $\mathbb A^{\ell}$ defined in \cite[Definition 4.5]{witttwists}. By \cite[Lemma 4.4]{witttwists} because $V_{\chi_1}$ is tamely ramified at infinity, and thus $\mathcal F_{\chi_1}$ is as well, the subset $U$ consists of the primitive Dirichlet characters.

To check this Betti number boundedness we must dig into the weeds of \'{e}tale cohomology. First note that the associated sheaf $\mathcal V$ is defined in \cite[\S6]{witttwists} as the composition of $V$ with the monodromy representation of a certain lisse sheaf $\mathcal G$. Because this composition is compatible with direct sums and tensor products, and Betti numbers are additive in direct sums, so we can reduce from $\rho$ to any representation of which $\rho$ is a summand. Any representation of the projective unitary group is a summand of the tensor product of $m$ copies of the standard representation of the usual unitary group with $m$ copies of the dual representation, so that is what we will take. In the case where $V$ is the tensor product of $m$ copies of the standard representation and its dual, by definition $\mathcal V$ is the tensor product of $m$ copies of the sheaf $\mathcal G$ with $m$ copies of the its dual. The sheaf $\mathcal G$ arises in \cite[Definition 4.2]{witttwists} as $R^1pr_{2!} (pr_1^* \mathcal F_{\chi_1}  \otimes \mathcal L_{\textrm{univ}})$ for $\mathcal L_{\textrm{univ}}$ a certain lisse sheaf of rank one. We can check that the dual sheaf is $R^1pr_{2!} (pr_1^* \mathcal F^\vee \otimes \mathcal L_{\textrm{univ}}^\vee)$ because they are each lisse pure sheaves and have the same trace function. So it remains to bound the Betti numbers of

\[ H^{i+2m}_c \left( U_{\FF_q }, \left( Rpr_{2!} (pr_1^* \mathcal F \otimes \mathcal L_{\textrm{univ}}) \right)^{\otimes m} \otimes \left( Rpr_{2!} (pr_1^* \mathcal F^\vee \otimes \mathcal L_{\textrm{univ}}^\vee) \right)^{\otimes m}\right) .\]

Observe that $U$ is an open subset of $\mathbb A^{\ell} $, the parameter space of characters with Swan conductor $\leq \ell$, with complement $\mathbb A^{\ell-1}$, parameterizing characters with Swan conductor $\leq \ell-1$. By excision, the compactly supported Betti numbers of $U$ with coefficients in the complex $\left( Rpr_{2!} (pr_1^* \mathcal F \otimes \mathcal L_{\textrm{univ}}) \right)^{\otimes m}\otimes \left( Rpr_{2!} (pr_1^* \mathcal F^\vee \otimes \mathcal L_{\textrm{univ}}^\vee) \right)^{\otimes m}$ are at most the sum of the Betti numbers of these two spaces with coefficients in the same complex. Because the two cases are equivalent with an index shifted by one, it suffices to bound the Betti numbers of $\mathbb A^{\ell}$. 

To do this, we apply the K\"{u}nneth formula and the projection formula, reducing us to

\[ H^{i+2m}_c \left(  \left( \mathbb A^1_{\mathbb F_q} \right)^{2m } \times \mathbb A^{\ell}_{\FF_q} , pr_1^* ( \mathcal F^{\boxtimes m} \boxtimes \mathcal F^{\vee \boxtimes m}) \otimes  \mathcal L_{\textrm{univ}}^{\otimes m} \otimes \mathcal L_{\textrm{univ}}^{\vee \otimes m} \right).\]

Applying the projection formula again, this is 

\[ H^{i+2m}_c \left(  \left( \mathbb A^1_{\mathbb F_q} \right)^{2m } ,\mathcal F^{\boxtimes m} \boxtimes \mathcal F^{\vee \boxtimes m}) \otimes  R pr_{1!} \left(  \mathcal L_{\textrm{univ}}^{\otimes m} \otimes \mathcal L_{\textrm{univ}}^{\vee \otimes m}\right)  \right).\]

In \cite[proof of Lemma 2.7, second equation]{representation-theory-moments} \[R pr_{1!} \left(  \mathcal L_{\textrm{univ}}^{\otimes m} \otimes \mathcal L_{\textrm{univ}}^{\vee \otimes m}\right)  = i_! \mathbb Q_{\ell'} [ -2 \ell] (-\ell)\] where $i$ is the inclusion of the closed set $Z$ in $\left( \mathbb A^1_{\mathbb F_q} \right)^{2m } $ where the first $\ell$ elementary symmetric polynomials in the first $m$ variables equal the first $m$ elementary symmetric polynomials in the last $m$ variables.

By a final application of the projection formula, we end up with \[ H^{i+2m-2\ell }_c \left( Z,  i^* \left( \mathcal F^{\boxtimes m} \boxtimes \mathcal F^{\vee \boxtimes m}\right)  \right).\]

Now $Z$ is a closed set in $\mathbb A^{2m}$ defined by $\ell$ equations of degree at most $\ell$. Because $M$ is squarefree, over a field extension in which it splits, $\chi_1$ is a product of at most $\deg M$ tame characters ramified at one of the roots of $M$ and $\infty$, so $\mathcal F$ is a tensor product of at most degree $M$ tame character sheaves $\mathcal L_{\rho_j}(x-a_j)$, for $a_j$ the roots of $M$. 

The Betti numbers are now bounded by Theorem~12 of \cite{katzbetti}, with $N=2m$, $r= \ell$, $s= 2 m \deg m$, $\delta=0$, $d_i=i$ for $i$ from $1$ to $\ell$, $e_1,\dots,e_s=1$, $f=0$, $F_1,\dots,F_\ell$ the defining equations of $Z$, $G_1,\dots, G_s$ the linear functions $x-a_j$ in the $2m$ variables with $a_j$ the roots of $M$. The Betti number bound is now given by Katz as \[ 3 ( 4m \deg M + \ell + 2) ^{2m+ \ell}\] which has all the desired uniformity properties (noting that $m$ depends only on $\rho$).

\end{proof}

\begin{theorem}\label{thmequishortgauss}
Let $\ell \ge 3$, and if $\ell=3$ assume that the characteristic is not $2$ or $5$. For any $\chi \in G(R_{\ell,1})$ and $\Delta \in \FF_q^{\times}$, set
\begin{equation*}
\psi_{\Delta}(\chi) = \frac{\sum_{c \in \FF_q^{\times}} \chi(T^{\ell}+\Delta c^{\ell})}{\sqrt{q}}.
\end{equation*}
Let $\rho$ be an irreducible representation of $\mathrm{PU}(\ell- 1)$. Then there exists a positive constant $D(\rho)$, depending only on $\rho$, such that
\begin{equation*}
\left| \frac{\sum_{\chi \in G(R_{\ell,1}) \text{ primitive}} \Tr(\rho(\Theta_{\chi})) \psi_{\Delta}(\chi)}{\# \{ \chi \in G(R_{\ell,1}) :  \chi \text{ primitive  \}}} \right| \le \frac{D(\rho)}{\sqrt{q}}.
\end{equation*}
\end{theorem}

\begin{proof} Let $\operatorname{exp}$ be a fixed additive character of $q$. Because $x \mapsto \chi(T^\ell+ x)$ is an additive character of $\mathbb F_q$, it is $\operatorname{exp}(a_{\chi}x)$ for some $a_{\chi}$ in $\mathbb F_q$. From basic properties of characters, each $a_\chi$ occurs equally often, and $\chi$ is primitive if and only if $a_\chi\neq 0$. We have the Gauss sum relation \[\psi_{\Delta} (\chi) = \frac{\sum_{c \in \FF_q^{\times}} \operatorname{exp} ( a \Delta c^\ell) }{\sqrt{q}} =\frac{-1}{\sqrt{q} }+ \sum_{\substack{ \chi'\colon  \mathbb F_q^\times \to \mathbb C^\times \\ \chi'^{\ell}=1 \\ \theta \neq 1}} \frac{ G(\chi'^{-1},\psi)}{q}  \chi' (a_\chi \Delta).\] It follows that

\[ \frac{\sum_{\chi \in G(R_{\ell,1}) \text{ primitive}} \Tr(\rho(\Theta_{\chi})) \psi_{\Delta}(\chi)}{\# \{ \chi \in G(R_{\ell,1}) :  \chi \text{ primitive  \}}}  \] \[=   \sum_{\substack{ \chi'\colon \mathbb F_q^\times \to \mathbb C^\times \\ \chi'^{\ell}=1 \\ \theta \neq 1}} \frac{ G(\chi'^{-1},\psi)}{q}  \chi'(\Delta)  \frac{\sum_{\chi \in G(R_{\ell,1}) \text{ primitive}} \Tr(\rho(\Theta_{\chi})) \chi'( a_\chi) }{\# \{ \chi \in G(R_{\ell,1}) :  \chi \text{ primitive  \}}} -  \frac{\sum_{\chi \in G(R_{\ell,1}) \text{ primitive}} \Tr(\rho(\Theta_{\chi})) \frac{1}{ \sqrt{q} }}{\# \{ \chi \in G(R_{\ell,1}) :  \chi \text{ primitive  \}}}  \] 

  Because $|\Tr(\rho(\Theta_{\chi}))| \leq \dim \rho$, we have \[ \left|  \frac{\sum_{\chi \in G(R_{\ell,1}) \text{ primitive}} \Tr(\rho(\Theta_{\chi})) \frac{1}{ \sqrt{q} }}{\# \{ \chi \in G(R_{\ell,1})\}} \right|  \leq \frac{\dim \rho}{\sqrt q}.\] Because $|G(\chi'^{-1},\psi)| =\sqrt{q}$ and $|\chi'(\Delta)|=1$, it suffices to show that the averages
 
 \[ \left| \frac{\sum_{\chi \in G(R_{\ell,1}) \text{ primitive}} \Tr(\rho(\Theta_{\chi})) a_{\chi'}(\chi)}{\# \{ \chi \in G(R_{\ell,1}) :  \chi \text{ primitive  \}}} \right|  \leq \frac{ C(\rho)}{\sqrt{q}}\] for some constant $C(\rho)$ (and then take $D(\rho) = (\ell-1) C(\rho) + \dim \rho$.)
 
 First let us handle the case where $\rho$ is the trivial representation. In this case, the bound follows immediately from the fact that each $a_\chi$ occurs equally often, so all the non-trivial $\chi'$ cancel completely.
 
 Next let us handle the case where $\rho$ is non-trivial and $p> 2\ell+1$. In this case, as explained in \cite[Remark 6.3]{WVQKR} the characters $\chi$ are associated to Artin-Schreier sheaves arising from polynomials of degree $\ell$. From \cite[the formula in the proof of Lemma 6.1]{WVQKR} we can see that $a_\chi$ is simply the top degree term of the polynomial times $\ell$.
 
 In the proof of \cite[Theorem 8.2]{WVQKR}, Katz shows that the average of $\Tr(\rho(\theta_\chi))$ over all polynomials of degree $\ell$ with every term but the linear term fixed is $O(1/\sqrt{q})$ for all but a fraction of $\leq 1/q$ of possible fixed choices for the high-degree terms (in fact, the problematic leading terms occur only for $\ell \geq 5$, and then they occupy a fraction at most $q^{- (\ell-3)/2}$ of the possible leading terms). Since $a_\chi$ is constant on these sets of polynomials, the same cancellation holds for  $\Tr(\rho(\theta_\chi)) \chi'(a_\chi)$. Summing over all possible choices of leading terms, we get the desired bound.

Next let us handle the case where $p$ is small. Here we cannot use the argument of Katz as a black box and must do some geometry. However, all the geometry is only a minor variant of the geometry done by Katz.

Let $\operatorname{Prim}_\ell$ be the space of primitive characters defined by Katz. He defined a sheaf $L_{\operatorname{univ}}$ on $\operatorname{Prim}_\ell$ \cite[\S4]{WVQKR} whose Frobenius conjugacy class at a point corresponding to a character $\chi$ is $\theta_\chi$ \cite[Lemma 4.1]{WVQKR}. By composing its monodromy representation with $\rho$, we obtain a sheaf $\rho(L_{\operatorname{univ}})$ whose Frobenius trace at a point is $\Tr(\rho(\Theta_{\chi}))$.

Let us in addition define a sheaf whose Frobenius trace at a point is $\chi'(a_\chi)$. To do this, we check that $a_\chi$ is a polynomial function on $\operatorname{Prim}_{\ell}$. Let $\ell_0$ be the largest prime-to-$\ell$ divisor of $\ell$.

To check this, observe that in the isomorphism defined in \cite[\S2]{WVQKR} between  $(1 + T \mathbb F_q[T])/ (1 + T^{\ell +1} \mathbb F_q[t])$ and $\prod_{1\leq m,m \textrm{ prime to } p, m\leq n} W_{l(m,\ell)}(\mathbb F_q)$,  because $1+ x T^\ell$ is the Artin-Hasse exponential of $- x T^{\ell}$, it is sent to a product which is $0$ in every factor except $m= \ell_0$ and $(0,\dots,0,\ell_0 x)$ in the factor with $m=\ell_0$.  The character of this group associated to a tuple of Witt vectors is defined in \cite[\S3]{WVQKR} by elementwise multiplying Witt vectors, taking the trace to the Witt vectors of $\mathbb F_p$, and applying a character of $W_{l(m,\ell)} (\mathbb F_p) $. The product of the Witt vector $(0,\dots,0,\ell_0 x)$ with another Witt vector depends only on the first coordinate $a_0$ of that other Witt vector, and taking the trace and applying a character of $W_{v_p(\ell_0)+1} (\mathbb F_p)$ is the same as taking $\operatorname{exp}( \ell_0 x a_0)$. Hence we can take $a_\chi= \ell_0 a_0$.

Then $\chi'(a_\chi)$ is the trace function of the Artin-Schreier sheaf $\mathcal L_{\chi'}( a_\chi)$. Thus we must show cancellation in

\[\sum_{x \in \operatorname{Prim}_{\ell} (\mathbb F_q)} \operatorname{tr}( \operatorname{Frob}_q, \rho(L_{univ}),x ) \operatorname{tr}(\operatorname{Frob}_q,  \mathcal L_{\chi'}(a_\chi),x) .\]
By the Lefschetz fixed point formula, this is \[\sum_{i =0}^{2\ell} (-1)^i \operatorname{tr}( \operatorname{Frob}_q, H^i_c (\operatorname{Prim}_{n, \overline{\mathbb F}_q}, \rho(L_{univ}) \otimes \mathcal L_{\chi'}(a_\chi) ).\]

Because $\rho(L_{univ})$ and $\mathcal L_{\chi'}(a_\chi)$ are both pure of weight $0$, eigenvalues of Frobenius acting on $H^i$ have norm at most $q^{i/2}$. We will show that $H^{2l}$ vanishes, so each trace is at most $q^{ \ell -1/2}$ times the dimension of $H^i$. Next we will show that the dimensions of the $H^i$ are uniformly bounded. Thus the sum of traces will be $O( q^{\ell-1/2})$ and dividing by the denominator will be $O(q^{-1/2})$, as desired.

We handle vanishing of the top cohomology first. By \cite[Theorem 5.1]{WVQKR}, under these assumptions, the monodromy of $L_{\operatorname{univ}}$ is a subgroup of $GL_{\ell-1}$ which contains $SL_{\ell-1}$ and thus maps surjectively onto $PGL_{\ell-1}$. In particular, $\rho(L_{univ})$ is irreducible. Furthermore, $\mathcal L_{\chi'}(a_\chi)$ is lisse of rank one, so $\rho(L_{univ})\otimes \mathcal L_{\chi'}(a_\chi)$ is irreducible. If $\dim \rho \neq 1$ then $\rho(L_{univ})\otimes \mathcal L_{\chi'}(a_\chi)$ has an irreducible monodromy representation of dimension greater than one and so is non-trivial. If $\dim \rho=1$ then $\rho$ is the trivial representation and so these components are simply the monodromy representations of the Kummer sheaf, which are non-trivial because $a_\chi$ is an affine coordinate of $\operatorname{Prim}_\ell$ under its isomorphism with an open subset of affine space. Hence in all cases the monodromy representation is irreducible and non-trivial, so it has no monodromy invariants, and thus the top cohomology vanishes.

To bound the Betti numbers, we can simply observe that in each characteristic, the sheaves in question can be defined only in terms of $\ell,\rho,p$ and not the finite field $\mathbb F_q$, so their Betti numbers are independent of $\mathbb F_q$. Because there are only finitely many $p$ left to consider, the Betti numbers are bounded in terms only of $\ell,\rho$.

\end{proof} 

\section*{Acknowledgments}
We wish to thank Lior Bary-Soroker and Ze'ev Rudnick for comments on an earlier version of the manuscript. 

The research of OG was supported by the European Research Council under the European Union's Seventh Framework Programme (FP7/2007-2013) / ERC grant agreement n$^{\text{o}}$ 320755.

This research was partially conducted during the period WS served as a Clay Research Fellow, and partially conducted during the period he was supported by  Dr. Max R\"{o}ssler, the Walter Haefner Foundation and the ETH Zurich Foundation.

\bibliographystyle{abbrv}
\bibliography{references}

\begin{thebibliography}{10}

\bibitem{andrade2015}
J.~C. Andrade, L.~Bary-Soroker, and Z.~Rudnick.
\newblock Shifted convolution and the {T}itchmarsh divisor problem over
  {$\mathbb{F}_q[t]$}.
\newblock {\em Philos. Trans. Roy. Soc. A}, 373(2040):20140308, 18, 2015.

\bibitem{bary2014}
L.~Bary-Soroker.
\newblock Hardy-{L}ittlewood tuple conjecture over large finite fields.
\newblock {\em Int. Math. Res. Not. IMRN}, 2014(2):568--575, 2014.

\bibitem{bary2017}
L.~Bary-Soroker and J.~Stix.
\newblock Cubic twin prime polynomials are counted by a modular form.
\newblock {\em arXiv preprint arXiv:1711.05564}, 2017.
\newblock To appear in {C}anadian Journal of Mathematics.

\bibitem{bender2009}
A.~O. Bender and P.~Pollack.
\newblock On quantitative analogues of the {G}oldbach and twin prime
  conjectures over $\mathbb{F}_q[t]$.
\newblock {\em arXiv preprint arXiv:0912.1702}, 2009.

\bibitem{carlitz1932}
L.~Carlitz.
\newblock The {A}rithmetic of {P}olynomials in a {G}alois {F}ield.
\newblock {\em Amer. J. Math.}, 54(1):39--50, 1932.

\bibitem{carmon2015}
D.~Carmon.
\newblock The autocorrelation of the {M}\"{o}bius function and {C}howla's
  conjecture for the rational function field in characteristic 2.
\newblock {\em Philos. Trans. Roy. Soc. A}, 373(2040):20140311, 14, 2015.

\bibitem{carmon2014}
D.~Carmon and Z.~Rudnick.
\newblock The autocorrelation of the {M}\"{o}bius function and {C}howla's
  conjecture for the rational function field.
\newblock {\em Q. J. Math.}, 65(1):53--61, 2014.

\bibitem{castillo2015}
A.~Castillo, C.~Hall, R.~J. Lemke~Oliver, P.~Pollack, and L.~Thompson.
\newblock Bounded gaps between primes in number fields and function fields.
\newblock {\em Proc. Amer. Math. Soc.}, 143(7):2841--2856, 2015.

\bibitem{chowla1965}
S.~Chowla.
\newblock {\em The {R}iemann hypothesis and {H}ilbert's tenth problem}.
\newblock Mathematics and Its Applications, Vol. 4. Gordon and Breach Science
  Publishers, New York-London-Paris, 1965.

\bibitem{conrey2001}
J.~B. Conrey and S.~M. Gonek.
\newblock High moments of the {R}iemann zeta-function.
\newblock {\em Duke Math. J.}, 107(3):577--604, 2001.

\bibitem{effinger1991}
G.~W. Effinger and D.~R. Hayes.
\newblock {\em Additive number theory of polynomials over a finite field}.
\newblock Oxford Mathematical Monographs. The Clarendon Press, Oxford
  University Press, New York, 1991.
\newblock Oxford Science Publications.

\bibitem{gamburd2007}
A.~Gamburd.
\newblock Some applications of symmetric functions theory in random matrix
  theory.
\newblock In {\em Ranks of elliptic curves and random matrix theory}, volume
  341 of {\em London Math. Soc. Lecture Note Ser.}, pages 143--169. Cambridge
  Univ. Press, Cambridge, 2007.

\bibitem{hardy1923}
G.~H. Hardy and J.~E. Littlewood.
\newblock Some problems of `{P}artitio numerorum'; {III}: {O}n the expression
  of a number as a sum of primes.
\newblock {\em Acta Math.}, 44(1):1--70, 1923.

\bibitem{hast2018}
D.~R. Hast and V.~Matei.
\newblock Higher moments of arithmetic functions in short intervals: A
  geometric perspective.
\newblock {\em International Mathematics Research Notices}, page rnx310, 2018.

\bibitem{hayes1965}
D.~R. Hayes.
\newblock The distribution of irreducibles in {${\rm GF}[q,\,x]$}.
\newblock {\em Trans. Amer. Math. Soc.}, 117:101--127, 1965.

\bibitem{ingham1927}
A.~E. Ingham.
\newblock Some {A}symptotic {F}ormulae in the {T}heory of {N}umbers.
\newblock {\em J. London Math. Soc.}, 2(3):202--208, 1927.

\bibitem{ivic1997}
A.~Ivi\'{c}.
\newblock The general additive divisor problem and moments of the
  zeta-function.
\newblock In {\em New trends in probability and statistics, {V}ol. 4
  ({P}alanga, 1996)}, pages 69--89. VSP, Utrecht, 1997.

\bibitem{katzbetti}
N.~M. Katz.
\newblock Sums of {B}etti numbers in arbitrary characteristic.
\newblock {\em Finite Fields and Their Applications}, 7:29--44, 2001.

\bibitem{katz2013}
N.~M. Katz.
\newblock On a question of {K}eating and {R}udnick about primitive {D}irichlet
  characters with squarefree conductor.
\newblock {\em Int. Math. Res. Not. IMRN}, 2013(14):3221--3249, 2013.

\bibitem{WVQKR}
N.~M. Katz.
\newblock Witt vectors and a question of {K}eating and {R}udnick.
\newblock {\em International Mathematics Research Notices},
  2013(16):3613--3638, 2013.

\bibitem{keating2016}
J.~P. Keating and E.~Roditty-Gershon.
\newblock Arithmetic correlations over large finite fields.
\newblock {\em Int. Math. Res. Not. IMRN}, 2016(3):860--874, 2016.

\bibitem{keating2018}
J.~P. Keating and E.~Roditty-Gershon.
\newblock Corrigendum to: “{A}rithmetic correlations over large finite
  fields”.
\newblock {\em International Mathematics Research Notices}, page rny162, 2018.

\bibitem{linnik1963}
J.~V. Linnik.
\newblock {\em The dispersion method in binary additive problems}.
\newblock Translated by S. Schuur. American Mathematical Society, Providence,
  R.I., 1963.

\bibitem{pollack2008polynomial}
P.~Pollack.
\newblock A polynomial analogue of the twin prime conjecture.
\newblock {\em Proc. Amer. Math. Soc.}, 136(11):3775--3784, 2008.

\bibitem{pollack2008}
P.~Pollack.
\newblock Simultaneous prime specializations of polynomials over finite fields.
\newblock {\em Proc. Lond. Math. Soc. (3)}, 97(3):545--567, 2008.

\bibitem{rhin1972}
G.~Rhin.
\newblock R\'{e}partition modulo {$1$} dans un corps de s\'{e}ries formelles
  sur un corps fini.
\newblock {\em Dissertationes Math. (Rozprawy Mat.)}, 95:75, 1972.

\bibitem{rodgers2018}
B.~Rodgers.
\newblock Arithmetic functions in short intervals and the symmetric group.
\newblock {\em Algebra Number Theory}, 12(5):1243--1279, 2018.

\bibitem{rosen2002}
M.~Rosen.
\newblock {\em Number theory in function fields}, volume 210 of {\em Graduate
  Texts in Mathematics}.
\newblock Springer-Verlag, New York, 2002.

\bibitem{sarnak2016}
P.~Sarnak.
\newblock Three lectures on {M}\"obius randomness.
\newblock {\em available at
  \url{http://www.math.ias.edu/files/wam/2011/PSMobius.pdf}}, 2011.

\bibitem{witttwists}
W.~Sawin.
\newblock The equidistribution of {$L$}-functions of {W}itt vector dirichlet
  characters over function fields.
\newblock https://arxiv.org/abs/1809.05137, 2018.

\bibitem{representation-theory-moments}
W.~Sawin.
\newblock A representation theory approach to integral moments of
  {$L$}-functions over function fields.
\newblock https://arxiv.org/abs/1810.01303, 2018.

\bibitem{sawin2019}
W.~Sawin and M.~Shusterman.
\newblock On the {C}howla and twin primes conjectures over {$\mathbb{F}_q[t]$}.
\newblock {\em arXiv preprint arXiv:1808.04001}, 2019.

\bibitem{schmidt1976}
W.~M. Schmidt.
\newblock {\em Equations over finite fields. {A}n elementary approach}.
\newblock Lecture Notes in Mathematics, Vol. 536. Springer-Verlag, Berlin-New
  York, 1976.

\bibitem{weil1974}
A.~Weil.
\newblock {\em Basic number theory}.
\newblock Springer-Verlag, New York-Berlin, third edition, 1974.
\newblock Die Grundlehren der Mathematischen Wissenschaften, Band 144.

\end{thebibliography}

\end{document}